%% file: Main_text.tex
\documentclass{article}

\input{Definitions_and_Packages.tex}

\makenomenclature

\begin{document}


\title{Designing Day-Ahead Multi-carrier Markets for Flexibility: Models and Clearing Algorithms\footnote{The work presented in this paper was carried out within the MAGNITUDE project which received funding from the European Union's Horizon 2020 research and innovation programme under grant agreement No 774309. This paper and the results described reflect only the authors' view. The European Commission and the Innovation and Networks Executive Agency (INEA) are not responsible for any use that may be made of the information they contain. \protect\\  $^\dagger$ Authors contributed equally. \protect\\  Email addresses: shahab.shariattorbaghan@wur.nl (Shahab Shariat Torbaghan), mma@n-side.com (Mehdi Madani), sels.peter@gmail.com (Peter Sels), ana.virag@vito.be (Ana Virag), helene.lecadre@vito.be(Hélène Le Cadre), kris.kessels@vito.be (Kris Kessels), yuting.mou@vito.be (Yuting Mou).}}
\date{December 2020}

\author[1,3,4]{Shahab~Shariat~Torbaghan $^\dagger$} 
%

\author[2]{Mehdi~Madani $^\dagger$ } 

\author[2]{Peter~Sels}

\author[1,3]{Ana~Virag}

\author[1,3]{Hélène~Le~Cadre}

\author[1,3]{Kris~Kessels}

\author[1,3]{Yuting~Mou}

\affil[1]{VITO, Boeretang 200, 2400 Mol, Belgium}
\affil[3]{EnergyVille, ThorPark 8310, 3600 Genk, Belgium}
\affil[2]{N-SIDE, Boulevard Baudouin 1er, 25, 1348 Louvain-La-Neuve, Belgium}
\affil[4]{ETE, Wageningen University and Research, Bornse Weilanden 9, 6700 AA Wageningen, The Netherlands}

\maketitle
\begin{abstract}
There is an intrinsic value in higher integration of multi-carrier energy systems (especially gas and electricity), to increase operational flexibility in the electricity system and to improve allocation of resources in gas and electricity networks.  The integration of different energy carrier markets is challenging due to the existence of physical and economic dependencies between the different energy carriers.  We propose in this paper an integrated day-ahead multi-carrier gas,  electricity and heat market clearing which includes new types of orders and constraints on these orders to represent techno-economic constraints of con-version and storage technologies.  We prove that the proposed market clearing gives rise to competitive equilibria. In addition, we propose two decentralised clearing algorithms which differ in how the decomposition of the underlying centralised clearing optimisation problem is performed.  This has  implications  in  terms  of  the  involved  agents  and  their  mutual  information  exchange.   It  is proven  that  they  yield  solutions  equivalent  to  the  centralised  market  clearing  under  a  mild  assumption of sufficient number of iterations.  We argue that such an integrated multi-carrier energy market mitigates (spot) market risks faced by market participants and enables better spot pricing of the different energy carriers.  The results show that conversion/storage technology owners would suffer from losses and/or opportunity costs, if they were obliged to only use elementary orders. For the test cases considered in this article, sum of losses and opportunity costs could reach up to 13,000 \euro{}/day and 9,000 \euro{}/day respectively, compared with the case where conversion and storage orders are used.
\end{abstract}

\vspace{-0.25cm}
\textbf{Keywords:} Multi-carrier energy markets, order types, decomposition, forecast uncertainty, social welfare optimisation.


\mbox{}


\vspace{-1cm}

\section*{Notation and Abbreviations}

\noindent

Sets:
\begin{description}[style=multiline,leftmargin=1.5cm]
	\item[$\mathcal{C}$] Set of energy carriers, including electricity $e$, gas $g$ and heat $h$. 
	\item[$\mathcal{CM}$] Set of   cumulative constraints.
	\item[$\mathcal{CO}$] Set of conversion orders. 
	\item[$\mathcal{O}$] Set of orders. 
	\item[$\mathcal{O}_t^c$] Set of orders for carrier $c$ at time step $t$. 
	\item[$\mathcal{PR}$] Set of  pro-rata constraints.
	\item[$\mathcal{T}$] Set of market clearing time steps.
	\item[$\mathcal{TS}$] Set of storage orders constraints. 
\end{description}
\bigskip

Variables:

\begin{description}[style=multiline,leftmargin=1.5cm]
	\item[$e_{s,t}$] Level of energy for storage order $s$ at $t$.
	\item[$x_{o}$] Acceptance ratio $\in [0,1]$ of order $o$.
	\item[$x_{o}^{c1\rightarrow c2}$] Acceptance ratio of the conversion order $o$, converting carrier $c1 \in \mathcal{C}$ to carrier $c2 \in \mathcal{C}$.
	\item[$x^{in}_{s,t}$] For storage order $s$, the fraction of quantity $Q^{in}_{s,t}$ injected in the market at $t$ .
	\item[$x^{out}_{s,t}$] For storage order $s$, the fraction of quantity $Q^{out}_{s,t}$ extracted from the market at $t$.
	\item[$\pi^{c}_{t}$] Market clearing price for carrier $c$ at period $t$.   Dual variable of power balance constraints for carrier $c$ at period $t$.
\end{description}
\bigskip
\noindent

Parameters:

\begin{description}[style=multiline,leftmargin=2cm]
	\item[$E_{s}^{\max}$] Energy storage maximum capacity.	
	\item[$P_{o}$] Price [\euro /MWh] of order $o$.
	\item[$P_{co}^{c_1\rightarrow c_2}$] The total conversion cost associated with conversion order $co$ when converting carrier $c_1$ to $c_2$.
    \item[$P_{s}^{spread}$] Spread or cost to recover by the buy and sell operations linked to storage order $s$.
	\item[$Q_{o}$]  Quantity [MWh] of order $o$. $Q^{c}_{o,t} < 0$ for sell orders and $Q^{c}_{o,t} > 0$ for buy orders.
	\item[$Q_{co}^{c_1\rightarrow c_2}$]The total conversion capacity associated with conversion order $co$ when converting carrier $c_1$ to $c_2$.
	\item[$Q_{s,t}^{out}, Q_{s,t}^{in}$] Quantities [MWh] for the storage order $s$ respectively denoting the total amount of energy that can be bought  at period $t$  or sold back  at period $t$.
	\item[$\eta_{co}^{c_1\rightarrow c_2}$] The conversion efficiency associated with conversion order $co$ when converting carrier $c_1$ to $c_2$. Efficiency must be $ > 0$ but is allowed to be $>1$ e.g. to model heat pumps.
	\item[$\eta_{s}^{out}, \eta_{s}^{in}$] Efficiencies in [0;1] for the storage order $s$ respectively denoting the efficiency  when energy is extracted from the market and the efficiency when energy is injected back in the market.
\end{description}

\bigskip
\noindent
Abbreviations of market models and clearing algorithms:

\begin{description}[style=multiline,leftmargin=2.5cm]
	\item[\normalfont MC0]  Market clearing of decoupled energy-carrier markets, described by \eqref{model:basemodel_obj}.
	\item[\normalfont MC1-FULL] Market clearing of integrated energy-carrier markets featuring conversion orders (see \eqref{model:conversion_bids_model_obj_gen}-\eqref{model:conversion_bids_model_end}) , storage orders (see \eqref{model:storage_primal_1_313}-\eqref{model:storage_primal_end_313} ) and constraints (see \eqref{model:CC_pr_primal_obj}-\eqref{model:CC_pr_primal_end}, \eqref{model:conversion_bids_model_g-to-p-h_obj-bis}-\eqref{model:conversion_bids_model_g-to-p-h_eq_end-bis}).
	\item[\normalfont MC1] Centralised market clearing algorithm, directly optimising formulation \eqref{MD53:conversion_bids_model_obj_gen} - \eqref{model523:conversion_bids_model_eq_end} (MC1-FULL with only conversion orders to simplify the presentation of decentralised algorithms leading to MC2 and MC3)
	\item[\normalfont MC2] Decentralised market clearing algorithm based on Lagragian decomposition (auxiliary variables with equality constraints), see the Lagragian relaxation and Lagragian dual \eqref{Lagrangian-relax-1}, \eqref{Lagrangian-dual-1}.
	\item[\normalfont MC3] Decentralised market clearing algorithm based on Lagrangian relaxation of the balance conditions, see the Lagragian relaxation and Lagragian dual \eqref{Lagrangian-relax-2}, \eqref{Lagrangian-dual-2}.
	\item[\normalfont MES] \color{black} Multi-energy systems \color{black}
\end{description}

\printnomenclature

\section{Introduction}
\subsection{Motivation and Background}
Climate change and the subsequent environmental and societal impacts \textcolor{black}{such as sea level rise, extreme weather events (e.g., severe storms and heat waves) \cite{fritze2008hope}} have led scientific and regulatory communities to seek solutions \textcolor{black}{for cutting greenhouse gas emissions to net-zero  \cite{cesena2019flexibility}, \cite{pinson2017towards}}. 
Sector coupling, i.e. integrated planning and operation of multi-energy systems (MES) in which the techno-economic interactions between different energy carriers and the associated impacts are explicitly accounted for, is seen as a promising answer to cutting greenhouse gas emissions \cite{EUGreenDeal2019}.

The multiplicity of energy sources and their growing interactions on the wholesale level, in industrial applications, and at the retail level, prove that the design of well-integrated systems and markets is heavily needed in the future. More examples are given by the joint gas and electricity network development plans, arise of local energy communities, e.g. in physical form of local heat networks coupled with electricity systems, or the large scale integration of conversion technologies such as heat pumps. In that context, the present paper aims at providing both, a market framework for a perfect integration of day-ahead multi-carrier energy markets and also a framework to conduct quantitative analyses on sector coupling at the European level.

Operational and planning aspects of technical integration of multi-energy systems already received some attention from the scientific community \cite{mancarella2014mes, cho2014combined}. \textcolor{black}{Authors in \cite{geidl2007,geidl2007optimal}  introduce the energy hub concept. It is a generic framework that determines a conversion matrix to describe the input/output relationship between the energy production, consumption and transport in multi-energy setting. 
Using the energy-hub concept, Krause et al. \cite{krause2011multiple} propose an extended modelling framework that can be used to study the operational (e.g., optimal power flow, risk management) and planning (e.g., investment) studies.} 

Sector coupling has already shown to enable utilisation of cross-sectoral flexibility \cite{yazdani2018strategic, clegg2015integrated, ordoudis2019integrated, 7533455, Zhou2017equivalent, shahidehpour2005impact, biskas2016coupled, zlotnik2016coordinated, zhao2009security, doi:10.1063/1.3600761, Application2010Shahidehpour}, and hence can lead to higher integration of renewable energy sources or even decrease of network costs.
\textcolor{black}{Authors in \cite{clegg2015integrated} develop an integrated framework to study the potential impact of power to gas (P2G) on the operation of electricity and gas transmission networks considering the interdependence's that will result from large-scale integration of P2G between the gas and electricity systems. The study demonstrates the techno-economic implications of P2G programs for \textcolor{black}{both electricity and gas networks} including enabling reducing wind curtailment in the electricity network by storing \textcolor{black}{hydrogen} in the gas network, with almost no disruption to the operation of the gas network. Likewise \cite{Anubhav2020, ordoudis2019integrated} use stochastic optimisation framework to study the operation of an integrated electricity and gas system with line-pack. The study shows that the flexibility inherited in the gas network can be used to resolve network congestion in the electricity network and from there, facilitates the integration of renewable energy resources in the network.}

\textcolor{black}{Chen et al. in \cite{7533455} introduced an integrated gas and electricity model that considers the correlation between electricity supply from renewables (e.g., wind) and demand through conversion technologies (e.g., power-to-gas units) and their impact on the coordinated operation of gas and power systems. Reference  \cite{Zhou2017equivalent} investigate the effects of variations in gas pipeline pressure on the ramp rates of NGFPPs. References \cite{shahidehpour2005impact, biskas2016coupled} study the impacts of natural gas and power system network on power generation scheduling and vice versa. The two systems are shown to be strongly interdependent and therefore, greatly benefit if operated in full coordination. Considering different system stress scenarios, \cite{zlotnik2016coordinated} investigates the economic and technical impact of implementing different coordination levels between the gas and electricity networks. The authors show that an increased coordination provides economic efficiency and improves security of supply, especially under high stress operational conditions. Neglecting the gas flow physics, \cite{zhao2009security}  proposed a two-stage stochastic optimisation framework to investigate the unit commitment problem considering natural gas supply uncertainties. Following this work, Liu et al.  proposed a bi-level optimisation model that seeks to minimise the electricity production costs subject to natural gas feasibility constraints in \cite{doi:10.1063/1.3600761} and investigated coordinated scheduling of natural gas and power systems from the perspective of a joint operator in \cite{Application2010Shahidehpour}. }

Much less attention was paid to studying aspects of economic integration of multi-energy systems, for instance, (day-ahead) market organisation\footnote{Throughout this paper we will focus on day-ahead time frame for market-based sector coupling.} needed to ensure the economically optimal sector coupling. A market-based, economic integration of multi-energy systems is crucial to operate the multi-energy systems in a liberalised context and deliver on a promise of sector coupling. Besides a few exemptions that consider electricity markets, gas markets, and heating/cooling\footnote{For  the  sake  of  brevity,  from now on, we  only  refer  to  heat  systems/markets in the paper.} (regulated) markets, \cite{van2018benefits, Brolin2020, zhong2018auction, xie2016optimal}, the majority of literature considers coupling of electricity and gas markets, see e.g., \cite{zhao2018shadow, chen2019unit, CUI2016183, 7913721, chen2017clearing, wang2018equilibrium}. 
\textcolor{black}{Zhao et al. \cite{zhao2018shadow} studied the day-ahead gas and electricity market model in which the two markets are operated by different entities. The proposed model considers physical properties of both systems and accurately reflects the specifications of gas-flow transit.
A key feature of the model is that it requires only limited information exchange between the gas and electricity markets (namely, fuel prices, supply and demand information) while keeping the network data and customer information for each system confidential. Reference~\cite{chen2019unit} introduced a sequential, bilateral, multi-time period electricity and gas market platform that models trades in the distribution networks. The second order convex relaxation technique is used to reduce computational costs of the nonlinear and non-convex electricity and gas equations. The author proposed a decomposition technique to \textcolor{black}{compute} the equilibrium. Market trades are shown to settle at the marginal prices which are determined at the equilibrium. 
Cui et al. \cite{CUI2016183} proposed an optimisation based scheduling framework to model coordinated operation of electricity and gas networks in a market-based environment. The scheduling framework is then used to study the bidding strategy of a virtual power plant \textcolor{black}{(VPP)} that seeks to maximise its profit by governing flexible load and coupon based demand response. The paper shows that several factors derives the VPP's effectiveness depends on several factors \textcolor{black}{first part of the sentence unclear to me} including the size of the system under study, the bidding strategy of the VPP and the underlying technology types.
\textcolor{black}{Introducing hierarchy in the decision process,} Wang et al. in \cite{7913721} study strategic behaviour of 
\textcolor{black}{market} participants in a synchronised integrated gas and electricity market that is operated by a centralised market operator. The market allows for bi-directional gas and electricity trading. The market clearing problem is formulated as an equilibrium problem with equilibrium constraints (EPEC) which is challenging to solve \cite{gabriel2012complementarity}. The synchronised mechanism is shown to be more efficient compared to the sequential clearing mechanism that is in practice to date. Chen et al. \cite{chen2017clearing} introduced a novel non-deterministic, coordinated gas and electricity market that results in more efficient allocation of resources, while considering the short-run uncertainty associated with renewable energies and imperfect price forecasts. The two systems are shown to be costly to schedule in real-time mainly due to long response time of gas to electricity technologies. To solve this problem, the authors introduce a new market product namely, reserved gas supply capacity which in essence is equivalent to curtailing low priority loads when the gas network capacity is needed for the power system. Jiang et al. \cite{wang2018equilibrium} propose a multi-period programming to study ``marginal price based bilateral energy trading on the equilibrium of coupled natural gas and electricity distribution markets". }

Some literature focuses on interactions between local heat and global electricity markets, see e.g. \cite{mitridati2016optimal, deng2019generalized}. \textcolor{black}{Reference \cite{mitridati2016optimal} assumes a local heat market and considers CHP units that participate in both electricity and heat markets. The authors model the economic dispatch problem of CHPs as a stochastic optimisation framework. Deng et al \cite{deng2019generalized} formulated the heat and electricity integrated market as an optimisation problem that seeks to maximise the social welfare \textcolor{black}{obtained as the sum} of the individual markets operators' utility functions. Using duality analysis, the authors derive a new pricing method, called generalised locational marginal pricing (GLMP) that accounts for coupling between the markets as well as the time-delay effects in the heating transfer process.}


A market-based integration of different energy carriers is challenging due to the existence of technical (e.g. temporal) and economic dependencies between the underlying physical systems of different carriers \cite{Brolin2020}, \cite{Kessels2019Innovative}. A key missing component here is a market-based coordination mechanism that is specifically designed to incorporate these dependencies,  \textcolor{black}{and as such can help integrating renewable energy sources and more efficient energy use, thus meeting the climate goals. Ideally, the market-based coordination of multi-energy systems on the one hand creates added value for the system as a whole, e.g. in form of increased social welfare or increased system capability to integrate renewables. On the other hand, ideally, such a market offers tools to market participants to communicate their own techno-economic constraints and offers them tools for risk mitigation due to uncertain market price realisations. 
Defining and} implementing such a  market-based mechanism and underlying orders and constraints  is a complex endeavour that requires diligent considerations \cite{Brolin2020, clegg2019integrated, kriechbaum2017smartexergy, lecadre2019}.

The above cited literature (\cite{van2018benefits} - \cite{deng2019generalized}) investigates the linkages and the associated impacts that exist between the energy systems of heat, natural gas and electricity (power) and \textcolor{black}{to a certain extent} study different levels of \textcolor{black}{market-based} coordination between the associated carriers. Some  \textcolor{black}{authors take perspective of market participants} and study the optimal participation strategy of certain technologies in such an integrated market environment, see e.g. \cite{CUI2016183, 7913721}. However, there is still need for further research in the field. 

For one, little attention is given to multi-carrier market integration and its potential impacts  on the market participant. Take conversion technologies such as gas-fired power plants, as an example. The economic success of such units strongly depends on the forecast accuracy of market prices in the source and destination carrier markets or more generally on the design of efficient bidding strategies facing price uncertainty on different carrier markets. One can see that simply integrating the market of different carriers would expose such technologies to substantial risks due to the biases introduced by the forecasts. 

As another research gap, very little has been done on different market designs and clearing algorithms in a fully integrated multi-carrier energy market design, and subsequently, the plausible implications thereof on 
the operation of such markets in practice and price formation aspects. 

\textcolor{black}{Finally, let us note that today, heat, gas and electricity markets are fully decoupled and managed by different entities. For example, orders submitted to electricity markets are not linked in any manner to orders submitted in the gas markets. This paper proposes market models where electricity, heat and gas markets are co-optimised and where the co-optimisation leverages new bidding products such as conversion orders and associated constraints that the market participants can introduce to better represent their techno-economic constraints. The market models introduced aim at removing economic inefficiencies due to the lack of market-based coordination.}

\subsection{Contributions and Organisation}

\textcolor{black}{In this paper, we provide answers to the following questions, which to the best of our knowledge have not been yet addressed in the literature: \textbf{(i)} what are the market order types and constraints that allow for economically efficient sector coupling as defined above\footnote{A market which results in an increased level of sector coupling, and at the same time offer an instrument to the market participants to reduce exposure to risks caused by imperfect price forecasts on different markets}? \textbf{(ii)} How should a market clearing mechanism with such order types and constraints look like and what are the economical and practical implications of calculating the market outcome in a centralised or decentralised manner?} 

\textcolor{black}{To this end, we provide mathematical formulation and analysis of multiple new integrated multi-carrier gas, electricity and heat market clearing models. The market clearing models include new market order types and constraints allowing market participants to describe their technical constraints and cost structures in a context of integrated multi-carrier markets}\textcolor{black}{, allowing to couple these markets in an efficient and economically coherent way. The techniques used consist in linear programming models appropriately using variables and constraints to represent the orders linking the different carrier markets, and leveraging linear programming duality to show that market outcomes form competitive equilibria ensuring that all conversion (resp. storage) technology owners are satisfied. Main ingredients are developed in Section 2.} 
Orders include characteristics (features) such as quantity, price, location, conversion costs and efficiencies, etc. Constraints are defined as a set of statements that are used to reflect various dependencies that exist amongst orders of different carriers, or orders of different time periods. Orders and constraints as such are defined to enable market participants to reflect their technological constraints (e.g., ramp rates, etc.) and private economic valuations (e.g., bid reserve prices) in the clearing process and from there, limit their exposure to the risk associated with energy trading. 
The main contributions of this paper lies firstly in the introduction and modelling of novel advanced (namely conversion and storage) orders, secondly in the mathematical formulation of the centralised and decentralised multi-carrier day-ahead market \textcolor{black}{clearing with the introduced advanced} orders, and thirdly, in providing  economic analyses and interpretations of the proposed market outcomes. 

The remainder of this paper is organised as follows. Section \ref{section:ordertypes} introduces the mathematical formulation of novel advanced orders outlined above. Section \ref{section:marketdesigns} develops three market clearings based on (subsets of) the advanced orders introduced in Section \ref{section:ordertypes}. A set of numerical results is presented in Section~\ref{section:results}. Section~\ref{section:conclusion} concludes the paper.

\section{\textcolor{black}{Novel Advanced} Orders and Multi-Carrier Energy Market \textcolor{black}{Clearing}} \label{section:ordertypes}
We introduce here the mathematical formulation of \textcolor{black}{two novel} advanced orders\textcolor{black}{: conversion and storage orders. These orders are additions to the market to increase participation of conversion, flexible and less controllable technologies in the market, and so improve the technology neutrality. They }enable market participants to trade energy across the different carrier markets \textcolor{black}{or across different time steps} while taking into account their technical constraints and/or cost structures. \textcolor{black}{The novel advanced orders allow to the market participants to externalise a part of their risk related to the price forecast errors as will be discussed in detail in this section and illustrated by the case studies. The new order types offer a mean for risk mitigation related to the price forecasts to the market participants under perfect competition. }

\textcolor{black}{Next to the novel order types, two new constraints, namely pro-rata and cumulative constraint, are introduced in this section. Constraints can be used to impose an additional relation between two or more orders. They give market participants more flexibility in translating their techno-economic characteristics into offers to the market, and hence make market more accessible to a broader range of technologies. }

\subsection{Market Clearing with Elementary, Conversion and Storage Orders}\label{section:elementary_orders}

We \textcolor{black}{first briefly describe the }classic elementary orders, \textcolor{black}{and then introduce the novel} conversion and storage orders. \textcolor{black}{Conversion orders allow for buying/selling energy in one energy carrier depending on the realized market prices of another energy carrier at the same time instance. Storage orders apply to a single carrier and allow market participants to trade energy across periods of the market clearing horizon, depending on the energy price differences  at different time instances.}

\subsubsection{Elementary Orders}
An elementary order $o$ is the simplest type of order similar to the most basic order type encountered in practice. It specifies the \textit{quantity ($Q^{c}_o$)} (amount of energy \textcolor{black}{in carrier $c$} in MWh) and a \textit{limit price ($P^{c}_o$)} in \euro{}/MWh \textcolor{black}{for the same carrier $c$}. A supplier indicates via this limit price what is the minimum price that he is willing to accept to provide energy, and a consumer specifies via this limit price what is the maximum price he is willing to pay to consume energy. In competitive markets, these limit prices should correspond respectively to the marginal cost of the seller and marginal utility of the buyer.

The corresponding welfare maximisation problem \textcolor{black}{with exclusively elementary orders} is \textcolor{black}{defined} as follows.

\textcolor{black}{For brevity, }we adopt the following convention:  $Q^{c}_{o} < 0$ for sell orders, and $Q^{c}_{o} > 0$ for buy orders.

The decoupled \textcolor{black}{ energy carrier} market \textcolor{black}{clearing problem} pursues maximisation of the aggregated social welfare over every carrier market $c \in \mathcal{C}$  and takes on the following form for each carrier market $c \in \mathcal{C}$:

\begin{subequations}
\begin{align}
     \max_{x_{o}^c}& \sum_{o \in \mathcal{O}^{c}} P^{c}_{o} Q^{c}_{o} x^{c}_{o}, \label{model:basemodel_obj_g}\\
\text{s.t.}     
& \sum_{o \in \mathcal{O}^{c}} Q^{c}_{o} x^{c}_{o} = 0, & [\pi^{c}] \label{model:basemodel_eq1_g} \\
& x^{c}_{o} \leq 1, \ o \in \mathcal{O}^{c},  & [\sigma_o^{c}] \\ 
 &x^{c}_{o} \geq 0, \ o \in \mathcal{O}^{c},  \label{model:basemodel_eq_end_g} 
\end{align}
\end{subequations}

\textcolor{black}{where a variable $x_o^c  \in [0,1]$ is the optimal fraction of the total offered order quantity $Q^{c}_{o}$ that is accepted in the market. }

Conditions (\ref{model:basemodel_eq1_g}) are the \textcolor{black}{energy} balance constraints for \textcolor{black}{each of} the carriers. The optimal values of the associated dual variables $\pi^c$ provide the market price for each carrier. Let us note that these optimisation problems can be aggregated into one global equivalent welfare maximisation problem:

\begin{equation}
\label{model:basemodel_obj}
    \max_{x^{c}_{o}} \Big\{\sum_{c \in \mathcal{C}} \sum_{o \in \mathcal{O}^{c}} P^{c}_{o} Q^{c}_{o} x^{c}_{o}, \quad \text{ s.t. } \quad \text{(\ref{model:basemodel_eq1_g})-(\ref{model:basemodel_eq_end_g})}\ 
    \Big\}.    
\end{equation}
Indeed, this global welfare maximisation problem is easily separable into the individual problems per carrier \eqref{model:basemodel_obj_g}-\eqref{model:basemodel_eq_end_g} (written separately per energy carrier $c$) as there are no shared variables \eqref{model:basemodel_eq1_g} \textcolor{black}{or \textcolor{black}{coupling constraints, among carriers}}.

\color{black}
\subsubsection{Conversion orders}
In this Section, we focus on the \textcolor{black}{novel conversion order type, which creates a tool for market participants to mitigate the risks related to the price forecasts of different carriers, and a mean to couple the markets of different carriers.  The  conversion orders are first introduced on an example of a natural gas-fired power plant that participates in gas and electricity markets, and then generalised to any conversion technology.} 

A conversion technology owner, \textcolor{black}{such as owner of a natural gas-fired power plant,} essentially buys from one or multiple energy carrier markets and converts energy to be sold in other energy carrier markets. 
\textcolor{black}{The following model describes conversion orders submitted by a market participant able to convert a total capacity of $Q_{co}^{g\rightarrow e} >0$ [MWh] gas to electricity at a certain conversion price $P_{co}^{g\rightarrow e}$ [\euro{}/MWh], and with a physical efficiency given by $\eta_{co}^{g\rightarrow e}$, where $\eta_{co}^{g\rightarrow e} \in [0,1]$. In the notation above, $co$ is the index of the considered conversion order.} 
\textcolor{black}{The proposed conversion orders have to be included in the market clearing as a part of objective function and part of constraints. }
\textcolor{black}{The market clearing optimisation problem with the conversion order is:}
\begin{subequations}

\begin{align}
    \max_{x}&
     \sum_{o \in \mathcal{O}^{g}} P^{g}_{o} Q^{g}_{o} x^{g}_{o} + \sum_{o \in \mathcal{O}^{e}} P^{e}_{o} Q^{e}_{o} x^{e}_{o} + \sum_{o \in \mathcal{O}^{h}} P^{h}_{o} Q^{h}_{o} x^{h}_{o} - P_{co}^{g\rightarrow e}Q_{co}^{g\rightarrow e} x_{co}^{g\rightarrow e},
    \label{model:conversion_bids_model_obj}\\
\text{s.t.}    & \sum_{o \in \mathcal{O}^{g}} Q^{g}_{o} x^{g}_{o} +   Q_{co}^{g\rightarrow e}x_{co,t}^{g\rightarrow e} = 0, & [\pi^g] 
    \label{model:conversion_bids_model_eq1} \\
    & \sum_{o \in \mathcal{O}^{e}} Q^{e}_{o} x^{e}_{o}  
    -   \eta_{co}^{g\rightarrow e}Q_{co}^{g\rightarrow e} x_{co}^{g\rightarrow e} = 0,  & [\pi^e] 
    \label{model:conversion_bids_model_eq2}    \\
    & \sum_{{o \in \mathcal{O}^{h}}} Q^{h}_{o} x^{h}_{o} = 0, & [\pi^h]
    \label{model:conversion_bids_model_eq3}\\
    & x_{o}^{c} \leq 1,  o \in \mathcal{O}^{c}, c \in \mathcal{C}, & [\sigma_{o}^{c}]   \label{model:conversion_bids_model_eq4}\\
    & x_{o}^{g}, x_{o}^{e}, x_{o}^{h}, x_{co}^{g\rightarrow e} \geq 0.\label{model:conversion_bids_model_ge_eq_end}
\end{align}
\end{subequations}

\textcolor{black}{Condition \eqref{model:conversion_bids_model_eq1}-\eqref{model:conversion_bids_model_ge_eq_end} are respectively the gas, electricity and heat energy balance constraints. The term $Q_{co}^{g\rightarrow e} x_{co}^{g\rightarrow e}$ corresponds to the gas bought by the conversion technology, while the term $-   \eta_{co}^{g\rightarrow e}Q_{co}^{g\rightarrow e} x_{co}^{g\rightarrow e}$ in the electricity balance condition \eqref{model:conversion_bids_model_eq2} is the corresponding supplied amount of electricity taking into account conversion losses via the parameter $\eta_{co}^{g\rightarrow e}$. \footnote{The signs of these extra terms are explained by the fact that the convention for $Q$ is that demand is a positive number and so, a demand for conversion from gas to electricity is also a positive value for gas.}}

\textcolor{black}{Note that the optimisation problem defined by \eqref{model:conversion_bids_model_obj}-\eqref{model:conversion_bids_model_ge_eq_end} is not fully separable any longer due to existence of coupling variables $x_{co}^{g\rightarrow e}$, which appear in the objective and constraints, and are needed to handle the proposed conversion orders. The conversion orders are therefore only possible in the integrated multi-carrier markets.}

In the general case, a conversion order $co \in \myset{CO}$ specifies a carrier of origin $c_1$, a carrier of destination $c_2$ and the associated parameters $P_{co}^{c_1\rightarrow c_2}$ [\euro{}/MWh],  $Q_{co}^{c_1\rightarrow c_2}$ [MWh], and $\eta_{co}^{c_1\rightarrow c_2}$ [-], respectively the conversion price, the total conversion capacity, and the conversion efficiency.  The formulation of the integrated (single-period) market model with conversion orders is:

\begin{subequations}
\begin{align}
    \max_{x_{o}^{c}, x_{co}^{g\rightarrow e}} &
    \sum_{c \in \mathcal{C}} \sum_{o \in \mathcal{O}^{c}} P^{c}_{o} Q^{c}_{o} x^{c}_{o} - \sum_{ 
    \substack{co \in \mathcal{O}^{c_1\rightarrow c_2} \\ c_1, c_2 \in \myset{C}}} P_{co}^{c_1\rightarrow c_2}Q_{co}^{c_1\rightarrow c_2}x_{co}^{c_1\rightarrow c_2},  \label{model:conversion_bids_model_obj_gen} \\
 \text{s.t.}   &\sum_{o \in \mathcal{O}^{c_1}} Q^{c_1}_{o} x^{c_1}_{o} +  
    \sum_{  \substack{co \in \mathcal{O}^{c_1\rightarrow c_2} \\ c_2 \neq c_1}} Q_{co}^{c_1\rightarrow c_2} x_{co}^{c_1\rightarrow c_2}  \\
- & \sum_{ \substack{co \in \mathcal{O}^{c_2 \rightarrow c_1}\\c_2 \neq c_1}} \eta_{co}^{c_2\rightarrow c_1}  Q_{co}^{c_2\rightarrow c_1} x_{co}^{c_2\rightarrow c_1}  = 0, c_1  \in  \myset{C}, &  [\pi^{c_1}] \label{model:conversion_bids_balance_condition} \\
    &x_{co}^{c_1 \rightarrow c_2} \leq 1,  co \in \myset{O}^{c_1 \rightarrow c_2}, c_1,c_2 \in \myset{C}, &[\sigma_{co,t}^{c_1 \rightarrow c_2}]  \\ 
    &x^{c}_{o} \leq 1, o \in \myset{O}_{}^{c}, c \in \myset{C},& [\sigma_{o}^{c}]\\
    &x_{o}, x_{co}^{c_1 \rightarrow c_2} \geq 0.\label{model:conversion_bids_model_end}
\end{align}
\end{subequations}

\textcolor{black}{Let us note that in a multi-period market setting, any of these conversion orders and all associated parameters can be indexed by the time step $t$ at which the conversion order should be considered. A location can be specified with locations connected through a transmission network. They are omitted here to simplify notation.}

We now \textcolor{black}{show that} the market outcome of the integrated multi-carrier markets with conversion orders results in economically efficient (e.g., perfectly optimal) outcomes for market participants offering their conversion technologies' \textcolor{black}{capacity by means of conversion orders}.

\begin{proposition}[Special case of Theorem \ref{theorem:competitive-equilibrium} in appendix.]\textbf{Competitive equilibrium in a market clearing with conversion orders}
\label{theorem:conversion_bids_equlibrium_property}\\
Let us consider market prices $\pi^c$ obtained as optimal dual variables  of the \textcolor{black}{energy} balance constraints \eqref{model:conversion_bids_balance_condition} written for each carrier $c\in C$ \textcolor{black}{of the optimisation problem \eqref{model:conversion_bids_model_obj_gen}-\eqref{model:conversion_bids_model_end}. Then, the following statements hold:}

If $(\eta_{co}^{c_1\rightarrow c_2} \pi^{c_2}  - \pi^{c_1} - P_c^{c_1\rightarrow c_2}) > 0$, i.e. if the conversion \textcolor{black}{order} is profitable, then $x_{co}^{c_1\rightarrow c_2} = 1$, that is the \textcolor{black}{conversion} order is fully accepted.

If $(\eta_{co}^{c_1\rightarrow c_2} \pi^{c_2}  - \pi^{c_1} - P_{co}^{c_1\rightarrow c_2}) < 0$, i.e. running the conversion \textcolor{black}{order} would incur losses, then $x_{co}^{c_1\rightarrow c_2} = 0$, that is the \textcolor{black}{conversion} order is fully rejected.

If $(\eta_{co}^{c_1\rightarrow c_2} \pi^{c_2}  - \pi^{c_1} - P_{co}^{c_1\rightarrow c_2}) = 0$, i.e. the conversion \textcolor{black}{order} is making zero profit or losses, then $x_{co}^{c_1\rightarrow c_2} \in [0;1]$, that is the \textcolor{black}{conversion} order can be fully accepted, partially accepted, or rejected. 
\end{proposition}

\textcolor{black}{This means that for the market prices, obtained as optimal dual variables of the \textcolor{black}{energy} balance constraints \eqref{model:conversion_bids_balance_condition}, the matched \textcolor{black}{energy} volumes (primal decisions) by the market operator correspond to optimal decisions for the market participants facing these market prices and optimising their profit: given those prices, no market participant can be better off by choosing other matched volumes. These market prices hence support a Walrasian equilibrium (see e.g. \cite{Mas-Colell1995}).}

\subsubsection{Storage orders}
\textcolor{black}{In this Section, we introduce the novel storage order type, which creates a tool for market participants with limited energy content\footnote{In this paper, for brevity, we refer to storage technology, when we refer to technology with limited energy content, such as storage technology or flexible demand response technology, e.g. electric vehicles, heat pumps, etc.} to mitigate the risks related to the energy price forecasts at different times, and a way to extend the market to more easily integrate flexibility. }
A storage order is used  to  define  willingness  to  transport  energy  over  time  at a certain price, if the technical limitations of the market participant are satisfied. 

A storage order is defined for a \textcolor{black}{single} energy carrier.\footnote{\textcolor{black}{Note that in this light, the proposed novel storage order type could straightforwardly be introduced in the existing European day-ahead electricity markets.}} For the sake of conciseness,  we only consider storage orders for electricity, however, the models can easily be extended to include storage orders for any energy carrier.

A storage order $s$ \textcolor{black}{is defined by }the quantities $Q_{s,t}^{out}$ \textcolor{black}{[MWh]} of \textcolor{black}{energy} that can be \textcolor{black}{consumed from (}taken out of) the market \textcolor{black}{and the quantities} $Q_{s,t}^{in}$ \textcolor{black}{[MWh]} that can be injected (\textcolor{black}{produced, }sold back) in the market at time $t$. To be consistent with the sign conventions used above where the supply quantities are negative and demand quantities are positive, we use here the convention according to which $Q_{s,t}^{in}<0$ and $Q_{s,t}^{out}>0$. 
\textcolor{black}{Next to the energy quantities, a storage order is characterised by a}  parameter $P_{s}^{spread}$ \textcolor{black}{[\euro{}/MWh], which} represents an average cost or an average price spread to be recovered by the market participant to provide energy storage \textcolor{black}{services as specified by the storage order} $s$ . \textcolor{black}{Lastly, a maximal amount of energy , $E_{s}^{max}$ [MWh], the initial and final energy levels, $E_{s}^{i}$ [MWh],} charging efficiency parameter $\eta^{out}_{s}$ and a discharging efficiency parameter $\eta^{in}_{s}$ can be specified to account for losses due to the energy injection or withdrawal operations.

The market clearing problem \textcolor{black}{with the proposed storage order type }is defined as the following \textcolor{black}{social welfare maximisation problem}: 
\begin{subequations}
\begin{align}
    \max_{x} &\sum_{t \in \mathcal{T}}  \bigg[ \sum_{o \in \mathcal{O}_t^{e}} P^{e}_{o,t} Q^{e}_{o,t} x^{e}_{o,t} \bigg] -  \sum_{s \in \myset{TS}} \sum_{t \in \mathcal{T}}  \Big[ P_{s}^{spread} x_{s,t}^{out} Q_{s,t}^{out}  \Big], 
    \label{model:storage_primal_1_313} \\
  \text{s.t.}  & \sum_{o \in \mathcal{O}_t^e} Q^{e}_{o,t} x^{e}_{o,t}   =  - \sum_{s} Q_{s,t}^{out}  x_{s,t}^{out} -  \sum_{s} Q_{s,t}^{in} x_{s,t}^{in}\eta^{in}_{s} , \  t \in \myset{T}, & [\pi_{t}^e] \label{storage-order:elec-balance-constraint} \\
    & e_{s, t} =  e_{s,t-1} + Q_{s,t}^{in} x_{s,t}^{in} + Q_{s,t}^{out} x_{s,t}^{out}\eta_{s}^{out},  t \in \myset{T},s \in \myset{TS}, & [\sigma_{s,t}^{e}] \label{storage_balance_313} \\
    & e_{s, t} \leq E_{s}^{max}, t \in \myset{T},s \in \myset{TS},  & [\sigma^{max}_{s,t}] \label{storage_max_energy} \\
    & e_{s, 0} = e_{s, T} = E_{s}^{i}, s \in \myset{TS}, & [\sigma_{s,t}^{i}, \sigma_{s,t}^{f}] \label{initial_state_313} \\
    & e_{s, t}, x^{g}_{o,t}, x^{e}_{o,t}, x^{h}_{o,t}, x_{s,t}^{in}, x_{s,t}^{out} \geq 0,  \\
    &x^{e}_{o,t} \leq 1,  t \in \myset{T}, o \in \mathcal{O}_t^e, & [\sigma^{e}_{o,t}] \\
    & x^{in}_{s,t} \leq 1,  t \in \myset{T},s \in \myset{TS}, & [\sigma^{in}_{s,t}] \\
    & x^{out}_{s,t} \leq 1,  t \in \myset{T},s \in \myset{TS},  & [\sigma^{out}_{s,t}] \label{model:storage_primal_end_313} 
\end{align}
\end{subequations}
where the variables $x_{s,t}^{out}$ and $x_{s,t}^{in}$ correspond respectively to the fraction of the total quantity $Q_{s,t}^{out}$ of \textcolor{black}{energy} that can be \textcolor{black}{consumed from} the market or $Q_{s,t}^{in}$ that can be \textcolor{black}{produced} at time $t$.

Conditions \eqref{storage-order:elec-balance-constraint} are the \textcolor{black}{energy} balance constraints for the different periods of the market clearing horizon.\footnote{\textcolor{black}{Note that the convention according to which $Q_{s,t}^{in}<0$ and $Q_{s,t}^{out}>0$ is used, as explained above.}} Conditions (\ref{storage_balance_313}) are storage energy balance conditions written for each storage order and each period. \textcolor{black}{These conditions enable keeping track of the energy storage content, and hence guaranteeing that the minimal and maximal energy storage capacities are not overridden, as specified by equation \eqref{storage_max_energy}.} Condition (\ref{initial_state_313}) indicates that \textcolor{black}{amount of energy in} the storage at the end of the \textcolor{black}{market clearing} horizon $e_{s, T}$ should be equal to the initial \textcolor{black}{amount of energy} $e_{s, 0}$, which is also fixed to an external parameter value $E_{s}^{i}$ submitted by the market participant (cf. condition (\ref{storage_balance_313}) written for $t=0$ or $t=T$ respectively). Note that in practice, condition (\ref{initial_state_313}) should not be considered as an explicit constraint: instead, condition (\ref{storage_balance_313}) for $t=0$ or $t=T$ is written with $e_{s, 0}$, $e_{s, T}$ substituted by the parameter $E_{s}^{i}$. The dual variables $\sigma_{s,t}^{i}$ and  $\sigma_{s,t}^{f}$ in \eqref{initial_state_313} correspond to the initial respectively final state versions of the constraints. Note that such a storage order focuses on offering transportation of energy across time: market participants willing to buy or sell energy besides offering their storage capabilities can still post buy or sell orders besides their storage order, and specify various links between them via the constraints discussed in Section \ref{section:constraints}.

\textcolor{black}{We now show that the outcome of a market with storage orders results in economically efficient outcomes for participants that make use of the proposed storage orders.}

\begin{proposition}[Special case of Theorem \ref{theorem:competitive-equilibrium} in appendix.] \textbf{Competitive equilibrium in a market clearing with storage orders}\\
Given the market prices of electricity $\pi_{t}^e$ obtained as optimal dual variables of (\ref{model:storage_primal_1_313})-(\ref{model:storage_primal_end_313}), any optimal solution of this optimisation problem is such that the corresponding optimal values of $x_{s,t}^{in}, x_{s,t}^{out}$ are also solving the profit maximising optimisation problem of the individual market participant operating the storage order $s$ given the fixed market prices $\pi_t^e$:
\begin{subequations}
\begin{align}
    \max_{e_{s, t}, x_{s,t}^{out}, x_{s,t}^{in}} &\sum_t (-Q_{s,t}^{in}) x_{s,t}^{in} \eta_s^{in} (\pi_t^e) - \sum_t Q_{s,t}^{out} x_{s,t}^{out} (\pi_t^e + P_{s}^{spread}) \label{storage_individual_opt-1} \\
\text{\normalfont{s.t.}} ~  & x^{in}_{s,t} \leq 1, & [\sigma^{in}_{s,t}] \\
    & x^{out}_{s,t} \leq 1, &[\sigma^{out}_{s,t}] \\
    & e_{s, t} =  e_{s,t-1} + Q_{s,t}^{out}x_{s,t}^{out}\eta_{s}^{out} + Q_{s,t}^{in} x_{s,t}^{in}, & [\sigma_{s,t}^{e}] \\
    & e_{s, t} \leq E_{s}^{max}, & [\sigma^{max}_{s,t}] \\
    & e_{s, 0} = e_{s, T} = E_{s}^{i}, \\
    & e_{s, t}, x_{s,t}^{in}, x_{s,t}^{out} \geq 0. \label{storage_individual_opt-end-2}
\end{align}
\end{subequations}
\end{proposition}

Essentially, this proposition states that the acceptance decisions made by a market operator are, for given market prices, maximising the profits that a market participant using storage orders could make assuming infinite market depth (without worrying about the availability of orders to match the buy/sell operations), where the profits correspond to the revenues collected from the energy sold back in the market, minus the costs of procurement related to the energy bought at the different time steps, minus the operations costs (or spread) to recover, modelled by $P_{s}^{spread} Q_{s,t}^{out} x_{s,t}^{out}$.

Note that the design of the storage order allows to consider injections or withdrawals at multiple time periods, and in that case, it will still be guaranteed that for the prices obtained from the market clearing, the market participant would not prefer to make other decisions at the different time periods than those obtained in the market outcome.

Another interpretation of this result is given as follows: under assumptions of perfect competition, and given the proofs on the competitive equilibrium, using storage orders is the perfect tool to replace complex yet less efficient bidding strategies relying on elementary orders, for a flexible technology with limited energy content. 

\color{black}
As outlined above, conversion and storage orders enable to reflect operational restrictions and financial preferences of conversion and limited energy technologies in the market clearing. Nevertheless, there are specific technologies for which the conversion (storage) orders are insufficient to reflect operational restrictions that are needed to be reflected otherwise. 

Take two types of combined heat and power (CHP) technologies as an example: CHP with back-pressure technology and CHP with extraction-condensing technology \cite{deng2019generalized}. 
The feasible operational domain of a CHP with back-pressure technology is an increasing line with the slope defined as the power-to-heat ratio, i.e. a proportion between heat and electricity produced by converting an amount of input gas, as depicted in Figure~\ref{fig:cbc}. In contrast, the feasible operation domain of CHP unit with extraction-condensing technology is modelled as a polygon as depicted in Figure~\ref{fig:cec}. Neither conversion order nor storage order are capable of adequately representing these technical constraints. For this reason, in the following section, we introduce a set of novel constraints.

\textcolor{black}{Let us also note that in general, modelling different operation modes of a CHP, with different conversion costs \textcolor{black}{or efficiencies} for each operation mode, would require introducing exclusive constraints. \textcolor{black}{An exclusive constraint states that only one out of many orders constrained by an exclusive constraint can be accepted. The exclusive order type has already been defined and used in the current electricity markets in Europe under the name ``Block Orders in an Exclusive group''. A CHP operator could then use  exclusive constraints to couple a number of conversion orders, one per operation mode, to represent this techno-economic complexity of CHP operation. In this way}, market will choose which order, i.e. in which operation mode should CHP operate to maximise the social welfare. This topic, which is leading to questions on pricing in the presence of non-convexities introduced by the binary variables, is beyond the scope of the present paper.}
\color{black}
\begin{figure}[h]
\begin{subfigure}{0.4\textwidth}
\includegraphics[width=0.8\linewidth]{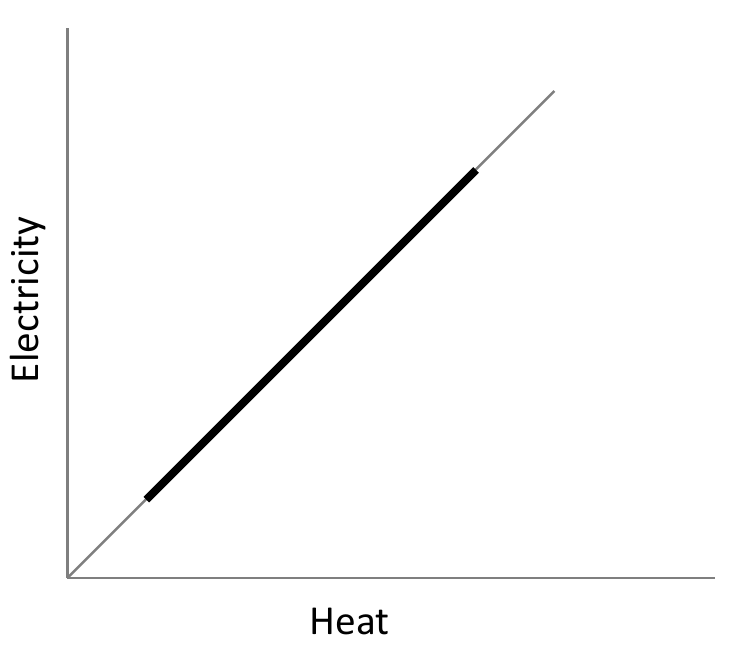} 
\caption{Back-pressure CHP technology.}
\label{fig:cbc}
\end{subfigure}
\begin{subfigure}{0.4\textwidth}
\includegraphics[width=0.8\linewidth]{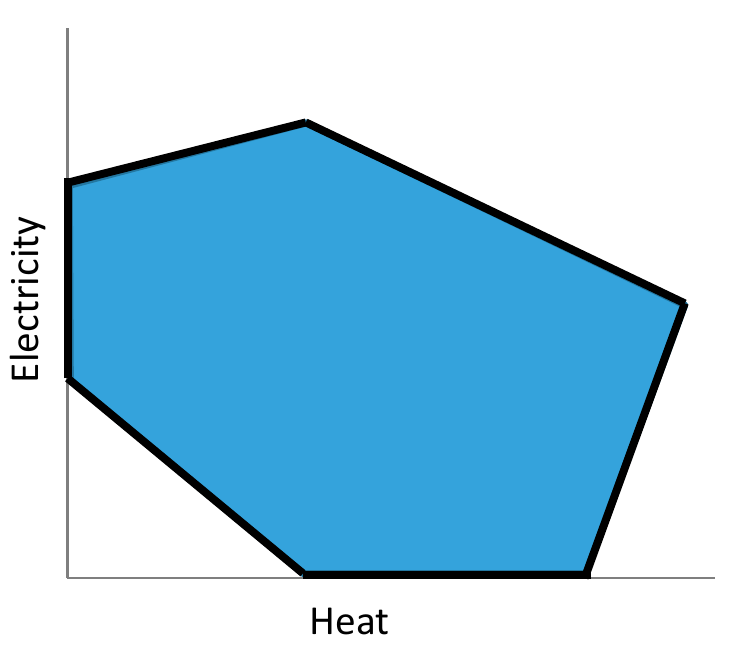}
\caption{Extraction-condensing CHP technology.}
\label{fig:cec}
\end{subfigure}
\caption{Feasible operation domain of CHP technologies. Figures are \textcolor{black}{inspired} from reference \cite{deng2019generalized}.}
\label{fig:CHP_tech}
\end{figure}

\subsection{Constraints on orders}\label{section:constraints}
We \textcolor{black}{propose two new} constraints that restrict how orders introduced above are cleared.  A constraint is a statement that makes the acceptance of an order conditional to the acceptance of another order in a predefined way.

\subsubsection{Pro-rata constraints} \label{continuous-constraints:pro-rata}
We \textcolor{black}{first introduce} pro-rata constraints which \textcolor{black}{enforce} that two orders should be accepted in same proportions. 

\color{black}
Pro-rata constraints are used by market participants by defining a set of orders to be constrained by a pro-rata constraint. 
\color{black}
The general form of pro-rata constraints is as follows:
\begin{align}
\label{constraint:Prorata}
    &x_{o} = x_{o^{'}}, &\forall pr=(o, o^{'}) \in \myset{PR},
\end{align}
where $pr \in \myset{PR}$ corresponds to the ordered pair of (elementary and/or conversion) orders to which the pro-rata constraint applies.

\color{black}
We illustrate a possible purpose of the pro-rata constraints and their inclusion in the market clearing on the following example of a back-pressure CHP (see Figure~\ref{fig:cbc}) with fixed electricity and heat output rates.

\color{black}

Let us assume that one can separate the cost of converting gas to electricity $P_{co}^{g\rightarrow e}$ from the cost of converting gas to heat $P_{co}^{g\rightarrow h}$ for the back-pressure CHP. \textcolor{black}{Obviously}, this is a simplifying assumption as in practice, splitting the cost of gas to heat and electricity conversions into two separated cost elements is not always a straightforward task. \textcolor{black}{However, let us note that today, bidders already have to proceed to such a split of costs even when such a split is not directly defined, to bid in the separate energy carriers markets, and without benefiting from the additional bidding tools that we introduce in this paper.}

\textcolor{black}{A market participant with such a back-pressure CHP, which is participating in a multi-carrier heat, gas and electricity market, could use two conversion orders, namely from gas to electricity and from gas to heat, constrained by a pro-rata constraint.}

The multi-carrier market clearing problem with elementary orders, conversion orders and pro-rata constraint on conversion orders is given as follows:
\begin{subequations}

\begin{equation}
    \max_{x \in X} \sum_{o \in \mathcal{O}^{g}} P^{g}_{o} Q^{g}_{o} x^{g}_{o} + \sum_{o \in \mathcal{O}^{e}} P^{e}_{o} Q^{e}_{o} x^{e}_{o} + \sum_{o \in \mathcal{O}^{h}} P^{h}_{o} Q^{h}_{o} x^{h}_{o} - \sum_c P_{co}^{g\rightarrow e} Q_{co}^{g\rightarrow e} x_{co}^{g\rightarrow e} - \sum_{co} P_{co}^{g\rightarrow h} Q_{co}^{g\rightarrow h} x_{co}^{g\rightarrow h},
    \label{model:CC_pr_primal_obj}
\end{equation}
\begin{align}
 \text{s.t.}   & \sum_{o \in \mathcal{O}^{g}} Q^{g}_{o} x^{g}_{o} = - \sum_{co} Q_{co}^{g\rightarrow e} x_{co}^{g\rightarrow e} - \sum_{co} Q_{co}^{g\rightarrow h} x_{co}^{g\rightarrow h}, & [\pi^g] 
\label{model:conversion_bids_model_g-to-p-h_eq1} \\
    & \sum_{o \in \mathcal{O}^{e}} Q^{e}_{o} x^{e}_{o} = \sum_{co} \eta_{co}^{g\rightarrow e} \sum_{co} Q_{co}^{g\rightarrow e} x_{co}^{g\rightarrow e}, & [\pi^e] \label{model:conversion_bids_model_g-to-p-h_eq2} \\
    & \sum_{o \in \mathcal{O}^{h}} Q^{h}_{o} x^{h}_{o} =  \sum_{co} \eta_{co}^{g\rightarrow h} \sum_{co} Q_{co}^{g\rightarrow h} x_{co}^{g\rightarrow h}, & [\pi^h] \label{model:conversion_bids_model_g-to-p-h_eq3} \\
    & x_{co}^{g\rightarrow e} \leq 1, co \in \mathcal{O}^{g\rightarrow e},   \\ 
    & x_{co}^{g\rightarrow h} \leq 1, co \in \mathcal{O}^{g\rightarrow h},   \\ 
    & x_{co_1}^{g\rightarrow e} = x_{co_2}^{g\rightarrow h},   (co_1,co_2) \in \myset{PR}, \label{prorata-cond_conversion} \\
    & x^{g}_{o}, x^{e}_{o}, x^{h}_{o} \leq 1,  o \in \mathcal{O}^{g} \cup \mathcal{O}^{e} \cup \mathcal{O}^{h},& [\sigma_o] \\\label{model:CC_pr_primal_end}
    &x^{g}_{o}, x^{e}_{o}, x^{h}_{o}, x_c^{g\rightarrow e},x_c^{g\rightarrow h} \geq 0.
\end{align}
\label{model:conversion_bids_model_g-to-p-h_eq_convcap}
\end{subequations}

The condition \eqref{prorata-cond_conversion} \textcolor{black}{is the implementation of the novel pro-rata constraint type in the market clearing,} and it guarantees that the amount of gas $Q_{co}^{g\rightarrow e} x_{co}^{g\rightarrow e}$ used to produce electricity \textcolor{black}{by the considered CHP} appearing in the gas balance condition \eqref{model:conversion_bids_model_g-to-p-h_eq1} is equal to the amount of gas used to produce heat $Q_{co}^{g\rightarrow h} x_{co}^{g\rightarrow h}$ \textcolor{black}{by the considered CHP} (the total amount of gas consumed is virtually "evenly split" between the electricity and heat production). Then, $\eta_{co}^{g\rightarrow e}$ and  $\eta_{co}^{g\rightarrow h}$ fix in which proportions electricity and heat can be produced per amount of gas, taking into account also the efficiency of each conversion process. In particular, $\eta_{co}^{g\rightarrow e}$ and  $\eta_{co}^{g\rightarrow h}$ are used to specify the fixed repartition of the total gas in input ($Q_{co}^{g\rightarrow e} x_{co}^{g\rightarrow e} + Q_{co}^{g\rightarrow h} x_{co}^{g\rightarrow h}$) between gas used to produce electricity ($\eta_{co}^{g\rightarrow e} Q_{co}^{g\rightarrow e} x_{co,t}^{g\rightarrow e}$), and gas used to produce heat ($\eta_{co}^{g\rightarrow h} Q_{co}^{g\rightarrow h} x_{co,t}^{g\rightarrow h}$), taking into account losses of conversion. 

By utilizing the conditions  (\ref{prorata-cond_conversion}), one thus indicates to the market that it is not possible to operate the conversion technology to produce electricity without producing heat as well, in certain fixed proportions specified by the parameters $\eta_{co}^{g\rightarrow e}$ and  $\eta_{co}^{g\rightarrow h}$.

\textcolor{black}{We now show that the market outcome of a market with elementary and conversion orders and pro-rata constraints (as defined by \eqref{model:CC_pr_primal_obj}-\eqref{model:CC_pr_primal_end}) results in economically efficient outcomes for market participants that make use of the proposed pro-rata constraints.}

\bigskip
\begin{proposition}[Special case of Theorem \ref{theorem:competitive-equilibrium} in appendix.] \textbf{Competitive equilibrium in a market clearing with pro-rata constraints}\\
Given the market prices $\pi^g, \pi^e, \pi^h$ obtained as optimal dual variables of (\ref{model:conversion_bids_model_g-to-p-h_eq1})-(\ref{model:conversion_bids_model_g-to-p-h_eq3}) in the market clearing problem defined by \eqref{model:CC_pr_primal_obj}-\eqref{model:CC_pr_primal_end}, and the conversion efficiencies $\eta_{co}^{g\rightarrow e}, \eta_{co}^{g\rightarrow h}$ and costs $P_{co}^{g\rightarrow e}, P_{co}^{g\rightarrow h}$, the primal solution $x^{g}_{o}, x^{e}_{o}, x^{h}_{o}, x_{co}^{g\rightarrow e}$, $x_{co}^{g\rightarrow h}$ given by the market operator also solves the profit maximising program of the market participant with a pro-rata constraint facing these market prices $\pi^g, \pi^e, \pi^h$:
\begin{subequations}
\begin{align}
\label{model:CC_prorata_obj}
\max_{x_{co}^{g\rightarrow e}, x_{co}^{g\rightarrow h}} & \Big((\eta_{co}^{g\rightarrow e} \pi^e - \pi^g) - P_{co}^{g\rightarrow e}\Big) Q_{co}^{g\rightarrow e} x_{co}^{g\rightarrow e}  + \Big((\eta_{co}^{g\rightarrow h} \pi^h  - \pi^g) -  P_{co}^{g\rightarrow h}\Big)Q_{co}^{g\rightarrow h} x_{co}^{g\rightarrow h}\\
\label{model:CC_prorate}
\text{\normalfont{s.t.}} ~ & x_{co}^{g\rightarrow e}, x_{co}^{g\rightarrow h} \leq 1,  \\
  & x_{co}^{g\rightarrow e} = x_{co}^{g\rightarrow h},  \\
\label{model:CC_endeq}
& x_{co}^{g\rightarrow e}, x_{co}^{g\rightarrow h} \geq 0.
\end{align}
\end{subequations}
\end{proposition}

\color{black}
\subsubsection{Cumulative constraints}\label{continuous-constraints:cumulative}
A cumulative constraint links two or more orders by imposing a limit on the sum of acceptance ratios of the orders that are constrained. The cumulative constraint implies that the orders can be accepted by any ratio, as long as the weighted sum of the acceptance ratios of the orders, is less or equal to 1. This can be used to specify a maximum amount of energy that the supplier is willing to supply. 

Cumulative constraints are used by market participants by defining a set of orders to be constrained by the constraint.  The general form of cumulative constraints is as follows: 
\begin{equation}
\label{constraint: cummulative}
    \sum_{o \in \mathcal{O}^{cm}} \vartheta^{cm}_{o} x_{o}  \leq 1, \ \ \ \ \ \  cm \in \myset{CM},
\end{equation}
where $\mathcal{O}^{cm}$ is the set of orders constrained by the cumulative constraint, $\mathcal{CM}$ is the set of cumulative constraints and $\vartheta^{cm}_{o}$ are coefficient factors that are provided by the market participants.

We illustrate inclusion of cumulative constraints in the market clearing on an example of a extraction-condensing CHP unit (see Figure~\ref{fig:cec}) where output ratios of electricity and heat are not fixed beforehand. 

The CHP converts gas to electricity at a cost of conversion $P_{co}^{g\rightarrow e}$ with a given physical efficiency given by $\eta_c^{g\rightarrow e}$, and converts gas to heat at a cost $P_{co}^{g\rightarrow h}$ with a given physical efficiency $\eta_{co}^{g\rightarrow h}$. The total conversion capacity is $Q_{co}^{g\rightarrow eh}$. The fact that the total amount of input gas $Q_{co}^{g\rightarrow eh}$ is split into two parts is modelled via the cumulative constraint (\ref{model:conversion_bids_model_g-to-p-h_eq_convcap-bis}). A market clearing with cumulative constraint and conversion orders is defined as the following social welfare maximisation problem:
\begin{subequations}
\begin{multline}
        \max_{x} \sum_{o \in \mathcal{O}^{g}} P_{o} Q_{o} x_{o} + \sum_{o \in \mathcal{O}^{e}} P_{o} Q_{o} x_{o} + 
    \sum_{o \in \mathcal{O}^{h}} P_{o} Q_{o} x_{o} \\
    - \sum_{co \in \mathcal{O}^{g\rightarrow e}} P_{co}^{g\rightarrow e} Q_{co}^{g\rightarrow eh} x_{co}^{g\rightarrow e} 
    - \sum_{co \in \mathcal{O}^{g\rightarrow h}} P_{co}^{g\rightarrow h} Q_{co}^{g\rightarrow eh}x_c^{g\rightarrow h},\label{model:conversion_bids_model_g-to-p-h_obj-bis}
\end{multline}

\begin{align}
 \text{s.t.}   & \sum_{o \in \mathcal{O}^{g}} Q_{o} x_{o} = - \sum_{co \in \mathcal{O}^{g\rightarrow e}} Q_{co}^{g\rightarrow eh}x_{co}^{g\rightarrow e} - \sum_{co \in \mathcal{O}^{g\rightarrow h}} Q_{co}^{g\rightarrow eh} x_{co}^{g\rightarrow h}, & [\pi^g] \label{model:conversion_bids_model_g-to-p-h_eq1-bis} \\
    & \sum_{o \in \mathcal{O}^{e}} Q_{o} x_{o} = \sum_{co \in \mathcal{O}^{g\rightarrow e}} \eta_{co}^{g\rightarrow e}Q_{co}^{g\rightarrow eh}x_{co}^{g\rightarrow e}, & [\pi^e] \label{model:conversion_bids_model_g-to-p-h_eq2-bis} \\
    & \sum_{o \in \mathcal{O}^{h}} Q^{h}_{o} x^{h}_{o} =  \sum_{co \in \mathcal{O}^{g\rightarrow e}} \eta_{co}^{g\rightarrow h}Q_{co}^{g\rightarrow eh}x_{co}^{g\rightarrow h}, & [\pi^h] \label{model:conversion_bids_model_g-to-p-h_eq3-bis} \\
  & x_{o} \leq 1 , o \in \mathcal{O}^{g} \cup \mathcal{O}^{e} \cup \mathcal{O}^{h},  &  [\sigma_o] \\
    & x_{co_1}^{g\rightarrow e} + x_{co_2}^{g\rightarrow h} \leq 1 , co_1, co_2 \in \myset{O}^{cm},   \label{model:conversion_bids_model_g-to-p-h_eq_convcap-bis} \\
    &x_{o}, x_{co}^{g\rightarrow e}, x_{co}^{g\rightarrow h} \geq 0.\label{model:conversion_bids_model_g-to-p-h_eq_end-bis}
\end{align}
    
\end{subequations}

We now show that the market outcome of this market results in economically efficient outcomes for market participants. 
\begin{proposition}[Special case of Theorem \ref{theorem:competitive-equilibrium} in appendix.] \textbf{Competitive equilibrium in a market clearing with cumulative constraints}\\
Given the market prices $\pi^g, \pi^e, \pi^h$ obtained as optimal dual variables of (\ref{model:conversion_bids_model_g-to-p-h_eq1-bis})-(\ref{model:conversion_bids_model_g-to-p-h_eq3-bis}) and the conversion efficiencies $\eta_{co}^{g\rightarrow e}, \eta_{co}^{g\rightarrow h}$ and costs $P_{co}^{g\rightarrow e}, P_{co}^{g\rightarrow h}$, the primal solution given by the market operator is optimal for the profit maximising program of the market participant with a cumulative constraint on the two conversions whose decision variables are $x_{co}^{g\rightarrow e}, x_{co}^{g\rightarrow h}$:
\begin{subequations}
\begin{align}
\max_{x_{co}^{g\rightarrow e}, x_{co}^{g\rightarrow h}} &  \Big((\eta_{co}^{g\rightarrow e} \pi^e - \pi^g) - P_{co}^{g\rightarrow e}\Big) Q_{co}^{g\rightarrow eh} x_{co}^{g\rightarrow e}  + \Big((\eta_{co}^{g\rightarrow h} \pi^h  - \pi^g) -  P_{co}^{g\rightarrow h}\Big) Q_{co}^{g\rightarrow eh} x_{co}^{g\rightarrow h},\\
\label{model:profitMaximization_chp_cum}
\text{\normalfont{s.t.}} ~   & x_{co}^{g\rightarrow e} + x_{co}^{g\rightarrow h} \leq 1, \\
    & x_{co}^{g\rightarrow e}, x_{co}^{g\rightarrow h} \geq 0.
\end{align}
\end{subequations}
\end{proposition}

The result of this theorem can be explained for the CHP technology in the example above as follows. By using the cumulative constraint, the owner of the CHP allows the market to decide which fraction of the input gas $Q_{co}^{g\rightarrow eh}$ will be used for the electricity production process, and which fraction will be used for the heat production process. Such repartition decisions will be made by the market operator in a way which is optimal for the market participant if it is price-taker, i.e. under the realized market prices, no other repartition results in higher profit for the CHP owner. 

A competitive equilibrium can also be obtained by maximising welfare when pro-rata and cumulative constraints are considered: they should then be considered both in the welfare maximisation problem and the individual profit maximising problems of market participants, as they also define their feasible set of decisions.

\section{Centralised and decentralised market clearing algorithms for the integrated multi-carrier energy markets} \label{section:marketdesigns}

Instead of the fully decoupled energy carrier markets, like the ones defined in Section \ref{section:elementary_orders} by \eqref{model:basemodel_obj}, energy carrier markets can be integrated via the use of conversion orders, or pro-rata and cumulative constraints coupling decision variables pertaining to different energy carrier markets, as explained in section \ref{section:ordertypes}. At least in theory and assuming no strategic behaviour of market participants, this design gives the most efficient market outcome: we come back to this point in Section \ref{sec:numerical-results:conversion-to-elementary} comparing markets with and without conversion (resp. storage) orders as bidding products for conversion (resp. storage) technology owners. Note that strategic behaviour of market participants is out of scope of this paper.

In this Section, we show how integrated multi-carrier energy markets with conversion orders and pro-rata and cumulative constraints can  be cleared in a centralised and in a decentralised manner\footnote{Note that a detailed classification of high-level possible multi-carrier energy market coordination was introduced in \cite{Kessels2019Innovative}, and the proposed market clearings which follow from the decentralised algorithms introduced in this section could be related to these market coordination schemes.}. In a centralised multi-carrier energy market clearing, a unique market operator (MO) clears the three carrier markets simultaneously. The centralised MO has full access to all the information about all orders.
In the decentralised market clearing, every carrier market is operated by a different carrier MO operating independent but exchanging information with the other MOs. 
In a decentralised multi-carrier energy market clearing, each carrier market has its own MO and limited information is exchanged between the involved parties. 

Two decentralised clearing algorithms are considered here, which differ in the agents or MOs involved and the data that they need to exchange. Both decentralised market clearing algorithms are obtained by considering a Lagrangian decomposition to solve the initial centralised, fully integrated multi-carrier energy market, where a Lagrangian dual is solved via a subgradient method \cite{shor1985minimization, nesterov2018lectures, anstreicher2009two, sherali1996recovery}. As we will see in more detail, it hence leads to results equivalent to those that would be obtained via a centralised market clearing. Note that other decomposition methods such as alternating direction method of multipliers (ADMM) \cite{boyd2011distributed,erseghe2014distributed} could have been considered. Lagrangian decomposition has been chosen here for its simplicity as it leads to an obvious separation of the optimisation problems to be solved by each market operator. Subgradient methods could be slower than other methods in practical cases but can also be more adequate on huge-scale problems or problems with special structure.

For the sake of simplicity of exposition, we consider elementary and conversion orders, but we do not explicitly consider storage orders and pro-rata or cumulative constraints in this section, although they could  straightforwardly be included in all developments which follow. 

\subsection{Integrated multi-carrier market with centralised clearing (MC1)}
We first consider the following centralised, integrated (gas, heat and electricity) market clearing problem which follows from the previous sections:

\begin{subequations}
\begin{multline}
\max_{x \in X} \Theta(x) :=
     \sum_{o \in \mathcal{O}^{g}} P^{g}_{o} Q^{g}_{o} x_{o} 
    + \sum_{o \in \mathcal{O}^{e}} P^{e}_{o} Q^{e}_{o} x_{o}
    + \sum_{o \in \mathcal{O}^{h}} P^{h}_{o} Q^{h}_{o} x_{o}\\
    - \sum_{co \in \mathcal{O}^{g\rightarrow e}} P_{co}^{g\rightarrow e} Q_{co}^{g\rightarrow e} x_{co}^{g\rightarrow e}
    - \sum_{co \in \mathcal{O}^{g\rightarrow h}} P_{co}^{g\rightarrow h} Q_{co}^{g\rightarrow h} x_{co}^{g\rightarrow h},
        \label{MD53:conversion_bids_model_obj_gen}    
\end{multline}

\begin{align}
\text{s.t.} ~    & \sum_{o \in \mathcal{O}^{g}} Q^{g}_{o} x^{g}_{o} +  \sum_{c \in \mathcal{O}^{g\rightarrow e}} Q_{co}^{g\rightarrow e} x_{co}^{g\rightarrow e} +  \sum_{co \in \mathcal{O}^{g\rightarrow h}} Q_{co}^{g\rightarrow h} x_{co}^{g\rightarrow h} = 0, & [\pi^g]
    \label{md523:power_balance_g} \\
    & \sum_{o \in \mathcal{O}^{e}} Q^{e}_{o} x^{e}_{o}  
    -  \sum_{co \in \mathcal{O}^{g\rightarrow e}} \eta_{co}^{g\rightarrow e} Q_{co}^{g\rightarrow e} x_{co}^{g\rightarrow e} = 0,  & [\pi^e] 
    \label{md523:power_balance_e} \\
    & \sum_{{o \in \mathcal{O}^{h}}} Q^{h}_{o} x^{h}_{o} -  \sum_{co \in \mathcal{O}^{g\rightarrow h}} \eta_{co}^{g\rightarrow h}Q_{co}^{g\rightarrow h} x_{co}^{g\rightarrow h} = 0, & [\pi^h]
    \label{md523:power_balance_h}\\
    & x_{o}^{g}, x_{o}^{e}, x_{o}^{h} \leq 1,  o \in \mathcal{O}^{g,e,h}, & [\sigma_{o}^{g}, \sigma_{o}^{h}, \sigma_{o}^{h}]   \label{model523:conversion_bids_model_eq4}\\
    & x_{co}^{g\rightarrow e} \leq 1,  c \in \mathcal{O}^{g \rightarrow e}, & [\sigma_{co}^{g\rightarrow e}]  \label{model523:conversion_bids_model_eq_convcap} \\
    & x_{co}^{g\rightarrow h} \leq 1 ,  co \in \mathcal{O}^{g \rightarrow h}, & [\sigma_{co}^{g\rightarrow h}]  \label{model523:conversion_bids_model_eq_convcap2} \\
    & x_{o}^{g}, x_{o}^{e}, x_{o}^{h}, x_{co}^{g\rightarrow e}, x_{co}^{g\rightarrow h} \geq 0.\label{model523:conversion_bids_model_eq_end}
\end{align}
\end{subequations}

The variables $x_{co}^{g\rightarrow e}, x_{co}^{g\rightarrow h}$ in constraints \eqref{md523:power_balance_g} - \eqref{md523:power_balance_h} prevent to split the optimisation problem into sub-problems per energy carrier that could be solved independently.

In case of the centralised market clearing, a single, centralised MO is in charge of operating the market.  It receives all the orders  and clears the market of all the carriers. It is a single-stage optimisation problem that maximises the aggregated social welfare of all the carrier markets.

Such a centralised multi-carrier energy market clearing with conversion orders is computationally attractive since it amounts to solving a (potentially large) linear program for which highly efficient methods and solvers are available. Prices can then be obtained either by explicitly solving the dual linear program, or obtained as optimal values of dual variables of the balance constraints. If price indeterminacies should be lifted (choosing  unique market clearing prices among a set of possibilities), a dedicated price problem to resolve these indeterminacies could also be considered. The feasible set of such a price problem is the set of optimal dual solutions to the initial welfare maximising primal problem, and an objective is chosen so as to select  unique, well-defined prices.

Note that here, the centralised MO has access to all information (including all orders, constraints and products), and clears each of the carrier markets considering directly all inter-dependencies that exist between the different carriers. 

In the next subsection, we consider two decentralised clearing algorithms in which each carrier market has its own MO and where coordination between the MO ensures that the iterative processes used to clear the markets in a decentralised way converge to results equivalent to results with a centralised market clearing (both sets of results are optimal solutions of the same welfare optimisation problem and correspond to a competitive equilibrium).

\subsection{Integrated multi-carrier market with decentralised clearing}
We consider in this section two decentralised clearing algorithms that both rely on solving the Lagrangian dual of the original centralised market clearing problem \eqref{MD53:conversion_bids_model_obj_gen}-\eqref{model523:conversion_bids_model_eq_end} via a subgradient optimization method allowing by construction a decentralisation of the computations among the different carrier market operators. They differ in which constraints are relaxed to construct the Lagrangian dual where the Lagrangian sub-problem can then be split into optimisation sub-problems separable per market operator (or `agent'). This has implications in terms of involved agents and exchange of information. The two decentralised market clearings are described in section~\ref{section:md5.2} and section~\ref{section:md5.3}.

Moreover, we provide precise mathematical statements regarding the convergence of solutions of the decomposition methods to an optimal solution of the centralised market clearing problem~\eqref{MD53:conversion_bids_model_obj_gen}-\eqref{model523:conversion_bids_model_eq_end}, by relying on relatively well-known results on primal convergence when sub-gradient optimisation methods are used to solve linear programs \cite{anstreicher2009two,sherali1996recovery}. (Note that recovery of primal solutions by means of averaging iterates is attributed to N. Shor.) Decentralised market clearing based on decomposition methods has been considered in the power systems and energy markets literature, see e.g. \cite{sorin2018consensus} and references therein. However, a rigorous convergence analysis is sometimes missing in this literature, in particular for the primal variables where the question of the convergence of iterates to a feasible and optimal primal solution is often not explicitly discussed, as highlighted in the reference \cite{sorin2018consensus}: ``it is common to use a dual convergence criterion only, with the understanding that primal and dual convergence are linked. Further work should be done to understand if it is also the case here". Let us note that primal converge is well-known for example for ADMM \cite{boyd2011distributed} as an alternative decomposition method. 

In our context, conversion orders play a pivotal role in the sense that, if there would be no conversion order, carrier markets would be considered decoupled and would be cleared separately. However, if there are conversion orders, there are shared (i.e., common) variables between carrier markets which can then not be cleared separately. We propose the two following approaches to deal with conversion orders (decompose the problem under existence of the common variables):
\begin{itemize}
    \item Approach 1 (decentralised multi-carrier market clearing with a single market operator per carrier market, MC2): introduce an extra auxiliary variable to present the acceptance ratio of a conversion order in the source and destination markets via two variables that are set to be equal, that is: $x^{c_1 \rightarrow c_2}_{co, t, c_1} = x^{c_1 \rightarrow c_2}_{co, t, c_2}$ where $c_1$ denotes the source market and $c_2$ the destination market. These constraints linking the duplicate conversion order variables are then 'dualised'. This approach  results in a setup where one MO is assigned to each carrier market. It is described in section~\ref{section:md5.2}. 
    \item Approach 2 (decentralised multi-carrier market clearing with an additional market operator for conversion orders, MC3): Dualise the balance constraints for each carrier market and define an auxiliary market operator, besides the market operators for each carrier market, that only processes conversion orders. In other words, there is one MO  assigned to each carrier market and an additional MO assigned to clear all the conversion orders. The approach is described in section \ref{section:md5.3}.
\end{itemize}

\subsubsection{Decentralised multi-carrier market clearing: a single market operator per carrier market (MC2)} \label{section:md5.2}
Let us first reformulate \eqref{MD53:conversion_bids_model_obj_gen} - \eqref{model523:conversion_bids_model_eq_end} by duplicating the variables $x_{co}^{g\rightarrow e}, x_{co}^{g\rightarrow h}$ that for now still prevent the optimisation problem to be split into sub-problems per energy carrier: $x_{co}^{g\rightarrow e}$ and $x_{co}^{g\rightarrow h}$ are respectively replaced in the gas balance constraint \eqref{model_5-2:pro-rata_ge}   by $x_{co,g}^{g\rightarrow e}$ and $x_{co,g}^{g\rightarrow h}$. Likewise $x_{co,e}^{g\rightarrow e}$ and $x_{co,h}^{g\rightarrow h}$ are respectively placed in the electricity balance \eqref{model_5-2:electricity-balance} and heat balance constraint \eqref{model_5-2:Part1} and the duplicate conversion order variables are required to take the same value via constraints \eqref{model_5-2:pro-rata_ge} and \eqref{model_5-2:pro-rata_gh}. 

The following is hence fully equivalent to \eqref{MD53:conversion_bids_model_obj_gen} - \eqref{model523:conversion_bids_model_eq_end}:
\begin{subequations}
\begin{align}
    \max_{x} \Omega(x) := &\sum_{o \in \mathcal{O}^{g}} P^{g}_{o} Q^{g}_{o} x^{g}_{o} + \sum_{o \in \mathcal{O}^{e}} P^{e}_{o} Q^{e}_{o} x^{e}_{o} + \sum_{o \in \mathcal{O}^{h}} P^{h}_{o} Q^{h}_{o} x^{h}_{o}  - \sum_{co \in \mathcal{O}^{g\rightarrow e}} P_{co}^{g\rightarrow e}Q_{co}^{g\rightarrow e} x_{co, e}^{g\rightarrow e} \notag \\
    & - \sum_{co \in \mathcal{O}^{g\rightarrow h}} P_{co}^{g\rightarrow h}Q_{co}^{g\rightarrow h} x_{co, h}^{g\rightarrow h},
    \label{model_5-2:conversion_bids_prorata_obj}
\end{align}

\begin{align}
 \text{s.t.} ~   & x_{co, g}^{g\rightarrow e} =  x_{co, e}^{g\rightarrow e},  co \in \mathcal{O}^{g\rightarrow e}_{}, & [\sigma_c^{g\rightarrow e}] \label{model_5-2:pro-rata_ge} \\
    &x_{co, g}^{g\rightarrow h} =  x_{co, h}^{g\rightarrow h}, co \in \mathcal{O}^{g\rightarrow h}_{},  & [\sigma_{co}^{g\rightarrow h}] \label{model_5-2:pro-rata_gh}  \\
    & \sum_{o \in \mathcal{O}^{g}} Q^{g}_{o} x^{g}_{o} = - \sum_{co \in \mathcal{O}^{g\rightarrow e}} Q_{co,g}^{g\rightarrow e} x_{co, g}^{g\rightarrow e} - \sum_{co \in \mathcal{O}^{g\rightarrow h}} Q_{co,g}^{g\rightarrow h} x_{co, g}^{g\rightarrow h}, & [\pi^g] \label{model_5-2:Part0}  \\
    & \sum_{o \in \mathcal{O}^{e}} Q^{e}_{o} x^{e}_{o} = \sum_{co \in \mathcal{O}^{g\rightarrow e}} Q_{co,e}^{g\rightarrow e} x_{co, e}^{g\rightarrow e}, & [\pi^e] \label{model_5-2:electricity-balance} \\
    & \sum_{{o \in \mathcal{O}^{h}}} Q^{h}_{o} x^{h}_{o} = \sum_{co \in \mathcal{O}^{g\rightarrow h}} Q_{co,h}^{g\rightarrow h} x_{co, h}^{g\rightarrow h}, & [\pi^h] \label{model_5-2:Part1}\\
\label{model_5-2:Part2}
    & x_{o}^{g}, x_{o}^{e}, x_{o}^{h}, x_{co, g}^{g\rightarrow e}, x_{co, e}^{g\rightarrow e}, x_{co, g}^{g\rightarrow h}, x_{co, h}^{g\rightarrow h} \leq 1, \\
    & x_{o}^{g}, x_{o}^{e}, x_{o}^{h}, x_{co, g}^{g\rightarrow e}, x_{co, e}^{g\rightarrow e}, x_{co, g}^{g\rightarrow h}, x_{co, h}^{g\rightarrow h} \geq 0. \label{model:conversion_bids_prorata_end}
\end{align}
\end{subequations}  

Equations~\eqref{model_5-2:pro-rata_ge} and \eqref{model_5-2:pro-rata_gh} are the additional constraints on duplicated variables constraints that link the duplicate acceptance variables for  the gas-to-electricity and gas-to-heat conversion orders, respectively. If these constraints are relaxed, the multi-carrier market clearing optimisation problem \eqref{model_5-2:conversion_bids_prorata_obj} - \eqref{model:conversion_bids_prorata_end}  can be split per energy carrier, i.e. per carrier MO. 
Let us relax the complicating constraints \eqref{model_5-2:pro-rata_ge} and \eqref{model_5-2:pro-rata_gh} by incorporating them into the objective function \eqref{model_5-2:conversion_bids_prorata_obj} using the  Lagrangian multipliers $\sigma_{co}^{g\rightarrow e}, \sigma_{co}^{g\rightarrow h} $:
\noindent
\begin{align}
    \Lagr(x, \sigma_{c}^{g\rightarrow e}, \sigma_{co}^{g\rightarrow h}) = &\Omega(x) + \sum_{co \in \mathcal{O}^{g\rightarrow e}} \sigma_{co}^{g\rightarrow e} (x_{co, g}^{g\rightarrow e} - x_{co, e}^{g\rightarrow e}) +  \notag\\ 
    &\sum_{co \in \mathcal{O}^{g\rightarrow h}_{}} \sigma_{co}^{g\rightarrow h} (x_{co, g}^{g\rightarrow h} - x_{co, h}^{g\rightarrow e})
    \label{model_5-2:LR_objective},
\end{align}

where $\Omega(x)$ is the objective function in \eqref{model_5-2:conversion_bids_prorata_obj}.

The relaxed problem is to maximise the Lagrangian function subject to constraints  \eqref{model_5-2:Part0}  - \eqref{model:conversion_bids_prorata_end}. It takes the following form: 
\noindent
\begin{equation}
    \phi(\sigma_{co}^{g\rightarrow e}, \sigma_{co}^{g\rightarrow h}) = \max_{0 \leq x \leq 1} \Big\{\Lagr(x, \sigma_{co}^{g\rightarrow e}, \sigma_{co}^{g\rightarrow h}),  \text{ s.t., } \text{ \eqref{model_5-2:Part0} - \eqref{model:conversion_bids_prorata_end}}\Big\}, \label{Lagrangian-relax-1}
\end{equation}
and the Lagrangian dual problem is
\noindent
\begin{equation}
    \min_{\sigma_{co}^{g\rightarrow e}, \sigma_{co}^{g\rightarrow h}}\max_{0 \leq x \leq 1} \Big\{\Lagr(x, \sigma_{co}^{g\rightarrow e}, \sigma_{co}^{g\rightarrow h}), \text{ s.t., } \text{ \eqref{model_5-2:Part0} - \eqref{model:conversion_bids_prorata_end}}\Big\}.    \label{Lagrangian-dual-1}
\end{equation}

One can observe that the right-hand side of \eqref{Lagrangian-relax-1}, where the Lagrangian multipliers $\sigma_{co}^{{g\rightarrow e}^{(k)}}, \sigma_{co}^{{g\rightarrow h}^{(k)}}$ are fixed parameters, can be split in optimisation problems per carrier market as follows:
\noindent
\small
\begin{subequations}
    \begin{align}
    \label{model52_decompoese_gas}
        &\max_{ \substack{x^{g}_{o}, x_{co,g}^{g\rightarrow e}\\ x_{co,g}^{g\rightarrow h}}} \Big\{ \Lagr^{g} + \sum_{co \in \mathcal{O}^{g\rightarrow e}} \sigma_{co}^{{g\rightarrow e}^{(k)}} x_{co,g}^{g\rightarrow e} + \sum_{co \in \mathcal{O}^{g\rightarrow h}} \sigma_{co}^{{g\rightarrow h}^{(k)}} x_{co,g}^{g\rightarrow h},~ \text{s.t.} \quad
        0 \leq x^{g}_{o}, x_{co,g} ^{g\rightarrow e}, x_{co,g}^{g\rightarrow h}\leq 1,  \eqref{model_5-2:Part0} \Big\}, \\
    \label{model52_decompoese_electricity}
        &\max_{x^{e}_{o}, x_{co,e}^{g\rightarrow e}} \Big\{\Lagr^{e} - \sum_{co \in \mathcal{O}^{g\rightarrow e}} \sigma_{co}^{{g\rightarrow e}^{(k)}} x_{co,e}^{g\rightarrow e},~ \text{ s.t. }  0 \leq x^{e}_{o}, x_{co,e}^{g\rightarrow e} \leq 1,  \eqref{model_5-2:electricity-balance}  \Big\}, \\
    \label{model52_decompoese_heat}
       & \max_{x^{h}_{o}, x_{co,h}^{g\rightarrow h}} \Big\{\Lagr^{h} - \sum_{co \in \mathcal{O}^{g\rightarrow h}_{co}} \sigma_{co}^{{g\rightarrow h}^{(k)}} x_{co,h}^{g\rightarrow h},~  \text{ s.t.}  \quad 0 \leq x^{h}_{o}, x_{co,h}^{g\rightarrow h} \leq 1,  \eqref{model_5-2:Part1} \Big\}, 
    \end{align}
\end{subequations} 
where
\begin{subequations}
    \begin{align}
        &\Lagr^{g} := \sum_{o \in \mathcal{O}^{g}} P^{g}_{o} Q^{g}_{o} x^{g}_{o},  \\
        &\Lagr^{e} := \sum_{o \in \mathcal{O}^{e}} P^{e}_{o} Q^{e}_{o} x^{e}_{o} - \sum_{co \in \mathcal{O}^{g\rightarrow e}} P_{co}^{g\rightarrow e}Q_{co}^{g\rightarrow e} x_{co, e}^{g\rightarrow e},\\
        &\Lagr^{h} :=\sum_{o \in \mathcal{O}^{h}} P^{h}_{o} Q^{h}_{o} x^{h}_{o} - \sum_{co \in \mathcal{O}^{g\rightarrow h}} P_{co}^{g\rightarrow h}Q_{co}^{g\rightarrow h} x_{co, h}^{g\rightarrow h}.
    \end{align}
\end{subequations}

It is well-known that strong duality holds for linear programs, that is we have:
\begin{align}
\label{model52:stepsize}
 \phi(\sigma_{c}^{{g\rightarrow e}^{(*)}}, \sigma_{co}^{{g\rightarrow h}^{(*)}}) = \Omega(x^{*}),
\end{align}
where $x^*$ solves \eqref{model_5-2:conversion_bids_prorata_obj} and $\sigma_{c}^{{g\rightarrow e}^{*}}, \sigma_{co}^{{g\rightarrow h}^{*}}$ solves \eqref{Lagrangian-dual-1}.

Note that problems \eqref{model52_decompoese_gas}, and \eqref{model52_decompoese_electricity} \eqref{model52_decompoese_heat} respectively present the gas, electricity and heat market clearing that will be cleared independently although iteratively, with some information exchange between the MOs as  described next.

Consider the following update rule for the multipliers \cite{anstreicher2009two}:
\begin{subequations}
    \begin{align}
        & \sigma_{co}^{{g\rightarrow e}^{(k+1)}} = \sigma_{co}^{{g\rightarrow e}^{(k)}} + \mu_{k} \Big(x_{co,g}^{{g\rightarrow e}^{(k)}} -  x_{co,e}^{{g\rightarrow e}^{(k)}}   \Big),\label{multipliers-update-1} \\
        & \sigma_{co}^{{g\rightarrow h}^{(k+1)}} = \sigma_{co}^{{g\rightarrow h}^{(k)}} + \mu_{k} \Big(x_{co,g}^{{g\rightarrow h}^{(k)}} -  x_{co,h}^{{g\rightarrow h}^{(k)}}  \Big), \label{multipliers-update-2}
    \end{align}

where $\mu_{k} \geq 0$ denotes a step-length sequence which satisfies the following three conditions:
    \begin{align}
        &\{\mu_{k} | \lim_{k\rightarrow +\infty} \mu_{k} = 0  \ \text{and} \  \sum_{k=1}^{\infty} \mu_{k} = \infty \ \text{and} \  \sum_{k=1}^{\infty} \mu_{k}^{2} < \infty \label{steplength-cond-3} \}.
    \end{align}
\end{subequations}

\begin{proposition}[Direct application of Theorem 2 in \cite{anstreicher2009two}]\label{proposition-LD-MD5.2}
Consider the iterative process where at each iteration $k$, the  $x_{co,g}^{{g\rightarrow e}^{(k)}}, x_{co,e}^{{g\rightarrow e}^{(k)}}, x_{co,g}^{{g\rightarrow h}^{(k)}}, x_{co,h}^{{g\rightarrow h}^{(k)}}$ are obtained by solving \eqref{model52_decompoese_gas}- \eqref{model52_decompoese_heat}, and the multipliers $\sigma_{c}^{{g\rightarrow e}^{k}}, \sigma_{co}^{{g\rightarrow h}^{k}}$ are updated according to \eqref{multipliers-update-1} and \eqref{multipliers-update-2} with the sequence $\mu_{k}$ satisfying \eqref{steplength-cond-3}. Then, $\sigma_{c}^{{g\rightarrow e}^{(k)}}\rightarrow \sigma_{c}^{{g\rightarrow e}^{*}}$ and  $\sigma_{co}^{{g\rightarrow h}^{(k)}} \rightarrow \sigma_{co}^{{g\rightarrow h}^{*}}$.
\end{proposition}

Regarding primal convergence, a well-known issue is that on the other side, the 'primal iterates' $x_{co,g}^{{g\rightarrow e}^{(k)}}$, $x_{co,e}^{{g\rightarrow e}^{(k)}}$, $x_{co,g}^{{g\rightarrow h}^{(k)}}$, $x_{co,h}^{{g\rightarrow h}^{(k)}}$ (together with the other values to form the iterates $x^k$ solutions of \eqref{model52_decompoese_gas}- \eqref{model52_decompoese_heat}) do not necessarily converge to an optimal solution of \eqref{model_5-2:conversion_bids_prorata_obj}. However, one can address this issue by adapting the iterates to build the following averaged iterates (see \cite{anstreicher2009two,sherali1996recovery}).
Let us consider $\mu_{k} = \frac{1}{k}$ as a particular step-length sequence satisfying \eqref{steplength-cond-3}, with 
\begin{equation}
\overline{x}^{(k)} = \frac{1}{k} \sum_{i=1}^k x^{(i)} \label{averaged-iterates}
\end{equation}
as the averaged iterates obtained from the $x^{(k)}$.

A direct application of Theorem 6 in \cite{anstreicher2009two} gives:
\begin{proposition}
\label{proposition-LD-MD5.2-b}
If in Proposition \ref{proposition-LD-MD5.2} the step-length is $\mu_{k} = \frac{1}{k}$, then any accumulation point of the sequence $\overline{x}^{(k)}$ obtained following the rule \eqref{averaged-iterates} is an optimal solution of \eqref{model_5-2:conversion_bids_prorata_obj}-\eqref{model:conversion_bids_prorata_end}, and hence, provides an optimal solution to \eqref{MD53:conversion_bids_model_obj_gen} - \eqref{model523:conversion_bids_model_eq_end}.

\end{proposition}

\paragraph{Information exchange in decentralised multi-carrier market clearing with a single market operator percarrier  market}

We analyse here the exchanges of information between the different energy carrier MOs that are required to execute the decentralised clearing algorithm just described.

The information exchange should be coordinated such that in every iteration $k$, the dual variables $\sigma_{co}^{{g\rightarrow e}^{k}}, \sigma_{co}^{{g\rightarrow h}^{k}}$ associated with the coupling equations \eqref{model_5-2:pro-rata_ge}, \eqref{model_5-2:pro-rata_gh} are communicated among the carrier market clearing agents (i.e., the sub-problems). To update these multipliers, one also needs to exchange the values of $x_{co,g}^{{g\rightarrow e}^{(k)}}, x_{co,e}^{{g\rightarrow e}^{(k)}}, x_{co,g}^{{g\rightarrow h}^{(k)}}, x_{co,h}^{{g\rightarrow h}^{(k)}}$.
 
Let us note that the clearing algorithm admits an economic interpretation where ``dummy demand gas orders" related to the variables $x_{co,g}^{{g\rightarrow e}^{}}, x_{co,g}^{{g\rightarrow h}^{}}$ are placed in the gas market, see the Lagrangian sub-problem \eqref{model52_decompoese_gas}. On the other side, ``dummy supply orders" are placed in the electricity and heat markets, respectively corresponding to the variables $x_{co,e}^{{g\rightarrow e}^{}}$ and $x_{co,h}^{{g\rightarrow h}^{}}$ in the Lagrangian sub-problems \eqref{model52_decompoese_electricity} and  \eqref{model52_decompoese_heat}. The market operators iteratively update the acceptance levels of these dummy orders until they reach a consensus on accepted volumes. Parallels could be made with consensus problems solved in a distributed way via ADMM, see Section 7 in \cite{boyd2011distributed}. 

The above analysis demonstrates that if an adequate information exchange is in place among different carrier market operators, and also if the stopping criteria and the step sizes are set adequately, a decentralised algorithm like this can in theory give optimal solutions to the same welfare optimisation problem, that could be otherwise obtained by directly solving \eqref{MD53:conversion_bids_model_obj_gen} - \eqref{model523:conversion_bids_model_eq_end} and its dual to get the market prices (note that optimal dual variables can be obtained for free with the simplex method or the interior-point methods implemented in commercial solvers). 

Note that  multiple optimal solutions to \eqref{MD53:conversion_bids_model_obj_gen} - \eqref{model523:conversion_bids_model_eq_end} could exist: the simplest example is given by a case without conversion orders, where supply and demand curves in a given carrier market overlap horizontally (orders with same prices), in which case multiple volume acceptances could be selected while providing the same total welfare (situation of ``volume indeterminacy").  In such cases, it is not a priori  guaranteed that the decentralised algorithm and the centralised algorithm will provide the exact same optimal solutions, but they will be equivalent in the sense that both will be welfare optimal and will correspond to a market equilibrium. Incorporating mechanisms to lift indeterminacies in the cleared volumes or in the market prices seem on another hand to require a centralised clearing or at least, in a decentralised setting, more advanced coordination mechanisms.

\textcolor{black}{Let us observe that market clearing prices are not explicitly available once the market clearing algorithm converged. However, further coordination among MOs enables to recover them, for example by exchanging information on bounds on prices in each carrier market (based on the acceptance of elementary orders and without revealing details about them) and by taking into account the shared information on conversion orders. In that case, each carrier market can easily solve on its side the same pricing problem where prices are variables constrained to be within the bounds communicated by each MO, and satisfy conditions ensuring that they are consistent with the acceptance of the conversion orders, knowing that data on conversion orders is accessible to all MOs.} If a further coordination among MOs to recover the prices is infeasible, another decentralised multi-carrier market clearing with an additional conversion order market operator next to the single market operators per carrier should be considered, as explained next. 

\subsubsection{Decentralised multi-carrier market clearing: a single market operator per carrier market and an additional conversion order market operator (MC3)}
\label{section:md5.3}

We consider here a second approach for decentralised multi-carrier market clearing \eqref{MD53:conversion_bids_model_obj_gen} - \eqref{model523:conversion_bids_model_eq_end}, in which we dualize the balance constraints \eqref{md523:power_balance_g}-\eqref{md523:power_balance_h}.

For fixed Lagrangian multipliers $\pi^g, \pi^e, \pi^h$, the Lagrangian relaxation obtained is:
    \begin{align}
        \phi(\pi^g, \pi^e, \pi^h) := \max_{0 \leq x \leq 1} \Lagr(x, \pi^g, \pi^e, \pi^h) = \Theta(x) &- \pi^g \Big(\sum_{o \in \mathcal{O}^{g}} Q^{g}_{o} x^{g}_{o} +  \sum_{co \in \mathcal{O}^{g\rightarrow e}} Q_{co}^{g\rightarrow e} x_{co}^{g\rightarrow e} +  \sum_{co \in \mathcal{O}^{g\rightarrow h}} Q_{co}^{g\rightarrow h} x_{co}^{g\rightarrow h}\Big) \notag \\
        &- \pi^e \Big( \sum_{o \in \mathcal{O}^{e}} Q^{e}_{o} x^{e}_{o} 
    -  \sum_{co \in \mathcal{O}^{g\rightarrow e}} \eta_{co}^{g\rightarrow e} Q_{co}^{g\rightarrow e} x_{co}^{g\rightarrow e}  \Big) \notag \\
        &- \pi^h \Big( \sum_{{o \in \mathcal{O}^{h}}} Q^{h}_{o} x^{h}_{o} -  \sum_{co \in \mathcal{O}^{g\rightarrow h}} \eta_{co}^{g\rightarrow h}Q_{co}^{g\rightarrow h} x_{co}^{g\rightarrow h}  \Big) , 
        \label{Lagrangian-relax-2}
    \end{align} 

where $\Theta(x)$ is the objective function in \eqref{MD53:conversion_bids_model_obj_gen}.

The Lagrangian dual to solve is
\begin{equation}
\min_{\pi^g, \pi^e, \pi^h}  \phi(\pi^g, \pi^e, \pi^h). \label{Lagrangian-dual-2}
\end{equation}

Similar to the discussion in the first decentralised clearing, at each iteration $k$, for given multipliers, $\pi^{g^k}$, $\pi^{e^k}$, $\pi^{h^k}$, the Lagrangian Relaxation problem to solve can be decomposed into separate sub-problems per energy market carrier and one last sub-problem of a ``conversion market operator (agent)" handling the sub-problem involving the conversion orders, i.e. conversion order variables $x_{co}^{g\rightarrow e}, x_{co}^{g\rightarrow h}$: 
\begin{subequations}
    \begin{align}
    \label{model53:decomposed_g}
        &\max_{x_o^g} \{\Lagr^g(x_o^g, \pi^{g^k}) | 0 \leq x_o^g \leq 1 \},\\
    \label{model53:decomposed_e}
        &\max_{x_o^e} \{\Lagr^e(x_o^e, \pi^{e^k})| 0 \leq x_o^e \leq 1 \}, \\
    \label{model53:decomposed_h}
        &\max_{x_o^h} \{\Lagr^h(x_o^h, \pi^{h^k}) | 0 \leq x_o^h \leq 1 \}, \\
    \label{model53:decomposed_con}
        & \max_{x_{co}^{g\rightarrow e}, x_{co}^{g\rightarrow h}} \{\Lagr^{g\rightarrow eh}( x_{co}^{g\rightarrow e}, x_{co}^{g\rightarrow h}, \pi^{g^k}, \pi^{e^k}, \pi^{h^k}  ) | 0 \leq x_{co}^{g\rightarrow e}, x_{co}^{g\rightarrow h} \leq 1  \} ,
    \end{align}
\end{subequations}
where 

\begin{subequations}
    \begin{align}
        &\Lagr^g(x_o^g, \pi^{g^k}) = \sum_{o \in \mathcal{O}^{g}} P^{g}_{o} Q^{g}_{o} x_{o}  - \pi^{g^k} \Big( \sum_{o \in \mathcal{O}^{g}} Q^{g}_{o} x^{g}_{o} \Big), \\
        &\Lagr^e(x_o^e, \pi^{e^k}) =  \sum_{o \in \mathcal{O}^{e}} P^{e}_{o} Q^{e}_{o} x_{o} - \pi^{e^k} \Big( \sum_{o \in \mathcal{O}^{e}} Q^{e}_{o} x^{e}_{o} \Big),
        \\
        &\Lagr^h(x_o^h, \pi^{h^k}) = \sum_{o \in \mathcal{O}^{h}} P^{h}_{o} Q^{h}_{o} x_{o} - \pi^{h^k} \Big( \sum_{{o \in \mathcal{O}^{h}}} Q^{h}_{o} x^{h}_{o} \Big),
    \end{align}  
    
    \begin{align}
    \label{model53:conversion_market_formula_obj}
        \Lagr^{g\rightarrow eh}( x_{co}^{g\rightarrow e}, x_{co}^{g\rightarrow h}, \pi^{g^k}, \pi^{e^k}, \pi^{h^k}  ) = &- \sum_{co \in \mathcal{O}^{g\rightarrow e}} P_{co}^{g\rightarrow e} Q_{co}^{g\rightarrow e} x_{co}^{g\rightarrow e} - \sum_{co \in \mathcal{O}^{g\rightarrow h}} P_{co}^{g\rightarrow h} Q_{co}^{g\rightarrow h} x_{co}^{g\rightarrow h} \notag \\
        & - \pi^{g^k} \Big( \sum_{co \in \mathcal{O}^{g\rightarrow e}} Q_{co}^{g\rightarrow e} x_{co}^{g\rightarrow e} +  \sum_{co \in \mathcal{O}^{g\rightarrow h}} Q_{co}^{g\rightarrow h} x_{co}^{g\rightarrow h}\Big) \notag \\
        & +\pi^{e^k} \Big( \sum_{co \in \mathcal{O}^{g\rightarrow e}} \eta_{co}^{g\rightarrow e} Q_{co}^{g\rightarrow e} x_{co}^{g\rightarrow e} \Big)          \notag \\
        &+ \pi^{h^k} \Big( \sum_{co \in \mathcal{O}^{g\rightarrow h}} \eta_{co}^{g\rightarrow h}Q_{co}^{g\rightarrow h} x_{co}^{g\rightarrow h} \Big).
    \end{align}
\end{subequations}

Consider now the following update rule for the multipliers, where again the $\mu_{k}$ satisfy the conditions \eqref{steplength-cond-3}:
\begin{subequations}
     \begin{align}
        & \pi^{{g}^{(k+1)}} = \pi^{{g}^{(k)}} + \mu_{k} \Big(\sum_{o \in \mathcal{O}^{g}} Q^{g}_{o} x^{{g}^{(k)}}_{o} +  \sum_{co \in \mathcal{O}^{g\rightarrow e}} Q_{co}^{g\rightarrow e} x_{co}^{g\rightarrow e^{(k)}} +  \sum_{co \in \mathcal{O}^{g\rightarrow h}} Q_{co}^{g\rightarrow h} x_{co}^{g\rightarrow h^{(k)}}  \Big),\label{multipliers-MD53-update-1} \\
        & \pi^{{e}^{(k+1)}} = \pi^{{e}^{(k)}} + \mu_{k} \Big(\sum_{o \in \mathcal{O}^{e}} Q^{e}_{o} x^{{e}^{(k)}}_{o} -  \sum_{co \in \mathcal{O}^{g\rightarrow e}} Q_{co}^{g\rightarrow e} x_{co}^{g\rightarrow e^{(k)}}\Big), \label{multipliers-MD53-update-2} \\
        & \pi^{{h}^{(k+1)}} = \pi^{{h}^{(k)}} + \mu_{k} \Big(\sum_{o \in \mathcal{O}^{h}} Q^{h}_{o} x^{{h}^{(k)}}_{o} -  \sum_{co \in \mathcal{O}^{g\rightarrow h}} Q_{co}^{g\rightarrow h} x_{co}^{g\rightarrow h^{(k)}}\Big). \label{multipliers-MD53-update-3}
    \end{align}
\end{subequations}

We then have the analogues to Proposition \ref{proposition-LD-MD5.2} and Proposition \ref{proposition-LD-MD5.2-b} for the first decentralised clearing algorithm in the previous section:
\begin{proposition}\label{proposition-LD-MD5.3}
Consider the iterative process where at each iteration $k$, the $x^{{g}^{(k)}}_{o}$, $x^{{e}^{(k)}}_{o}$, $x^{{h}^{(k)}}_{o}$, $x_{co}^{g\rightarrow e^{(k)}}$, $x_{co}^{g\rightarrow h^{(k)}}$  are obtained by solving \eqref{model53:decomposed_g}-\eqref{model53:decomposed_con}, and the multipliers $(\pi^{g^k}, \pi^{e^k}, \pi^{h^k} )$ are updated according to \eqref{multipliers-MD53-update-1}-\eqref{multipliers-MD53-update-3} with the sequence $\mu_{k}$ satisfying \eqref{steplength-cond-3}. Then, $(\pi^{g^k}, \pi^{e^k}, \pi^{h^k} ) \rightarrow (\pi^{g^*}, \pi^{e^*}, \pi^{h^*} )$.
\end{proposition}

\begin{proposition}\label{proposition-LD-MD5.3-b}
If in Proposition \ref{proposition-LD-MD5.3} the step-length is $\mu_{k} = \frac{1}{k}$, then any accumulation point of the primal iterate sequence $\overline{x}^{(k)}$ obtained following the rule \eqref{averaged-iterates} is an optimal solution of \eqref{MD53:conversion_bids_model_obj_gen} - \eqref{model523:conversion_bids_model_eq_end}.
\end{proposition}

\paragraph{Information exchange in decentralised market clearing with a single market operator per carrier market and an additional conversion order market operator}
The developments above show that the centralised market clearing problem can be decomposed into four sub-problems that are processed by four independent market operators (clearing agents): one per energy carrier, and one for conversion orders, which exchange a limited amount of information. 
Conversion orders are cleared in a separated trading platform, by a conversion order agent, which at each iteration solves the optimisation problems with conversion orders described by \eqref{model53:decomposed_con}, \eqref{model53:conversion_market_formula_obj}.

The LR problems sub-problems \eqref{model53:decomposed_g} - \eqref{model53:decomposed_con} are solved iteratively where at every iteration $k$, the updated Lagrangian multipliers $(\pi^{g^k}, \pi^{e^k}, \pi^{h^k} )$ are used. These sub-problems require that at each iteration, the involved agents (market operators) share information regarding the current values of $(\pi^{g^k}, \pi^{e^k}, \pi^{h^k} )$. 

To update these multipliers, according to \eqref{multipliers-MD53-update-1}-\eqref{multipliers-MD53-update-3}, each agent can add its contribution to the update of the $(\pi^{g^k}, \pi^{e^k}, \pi^{h^k})$ while keeping the detailed individual of values $x_o^{g^{(k)}}, x_o^{e^{(k)}}, x_o^{h^{(k)}}$ and $x_{co}^{g\rightarrow e^{(k)}}, x_{co}^{g\rightarrow h^{(k)}}$ private (note that each agent knows the index of the current iteration $k$ and hence $\mu_{k}$).

\subsection{Comparison of the market clearing algorithms}

We compare here the three proposed market clearing algorithms for integrated multi-carrier energy markets, MC1, MC2 and MC3 (cf. the section Notation and Abbreviations at the beginning of the paper for quick reference to the models involved).

Table \ref{table:comparing_MDs} summarises the characteristics of different market clearing algorithms, including some information on the computational complexity of each market clearing. Regarding the computational complexity, we refer to Chapter 15 in \cite{schrijver1998theory} for the polynomial complexity of linear programs, which corresponds to the complexity of MC1 (see also Chapter 2 therein for a definition of polynomial complexity). Note that as of today, no \emph{strongly} polynomial algorithm for linear programming is known \cite[Chapter 15]{schrijver1998theory}, but commercial solvers are able to efficiently solve very large-scale problems. Regarding the convergence rate of subgradient methods (used for MC2 and MC3), a detailed analysis is given in \cite{nesterov2018lectures}. Each iteration of MC2 and MC3 requires to solve a given number of linear programs, and therefore also has a complexity polynomial in the size of the instance. Let us highlight that in the context of MC3, solving the linear programs \eqref{model53:decomposed_g}-\eqref{model53:decomposed_con} can be very fast, since there is no balance constraint to satisfy, the only constraints are bounds on variables, and the optimal variable value for each order can quickly be be computed based on the current prices (value of the Lagrangian multipliers at a given iteration).

Despite the need for creating a common market operator for gas, heat and electricity carrier markets, MC1 has two advantages. 
First, it does not require an advanced algorithmic and IT cooperation among different market operators to operate the markets in a decentralised way. Second, it also enables to directly use state-of-the art commercial linear programming solvers to clear the markets, and also to implement algorithms to lift indeterminacies in accepted volumes or prices if any (in case of multiple optimal solutions to the volume matching primal, or pricing dual problems).

On another hand, the decentralised clearing MC2 is the closest to the current situation in terms of existing organisations and does not require to create a new entity for clearing the integrated energy carrier markets, such as a new central MO for all carrier markets as in MC1, or a new MO to handle conversion orders as in MC3. However, although MC2 is closer to the current market implementation in terms of involved organisations, this benefit would a priori be largely offset by the required evolution in systems and coordination required to clear the markets in such a decentralised way.

Let us note that the polynomial complexity for directly solving  MC1 is given in terms of the size of the instance to solve. On the other hand, the number of iterations to reach an optimal solution with precision $\epsilon$ with the subgradient method is given as a function of that precision. Each iteration requires to solve a linear program and therefore also has a polynomial complexity.  

\begin{table}[ht!]
\centering
\caption{Comparison of the Market Clearing Algorithms}
\label{table:comparing_MDs}
\begin{tabular}{||C{3cm}||p{4cm}|p{4cm}|p{4cm}||} 
 \hline \hline
   & \textbf{MC1} & \textbf{MC2} & \textbf{MC3}\\ 
 \hline\hline
    \textbf{Computational performance} &  - Number of operations polynomial  in the size of the market instance.
    & - $O(\frac{1}{\epsilon^2})$ iterations to solve the Lagrangian dual with precision $\epsilon$.  & - $O(\frac{1}{\epsilon^2})$ iterations to solve the Lagrangian dual (giving market prices) with precision $\epsilon$.    \\
     &  &  - Each iteration polynomial in the size of the  instance   & - Each iteration polynomial in the size of the  instance   \\    
\hline
    \textbf{Agents involved} & One central MO & Three MOs, one per carrier & Four MOs, one per carrier and one for conversion orders\\ 
\hline
   \textbf{Similarity with current European electricity market design} & Low & Intermediate & Low \\ 
\hline
   \textbf{Possibility to lift indeterminacies in volumes and prices}  & Yes & No & No \\ 
 \hline  \hline
\end{tabular}
\end{table}

To conclude, which market clearing procedure is superior among the three will depend on the  policy goals to be achieved, characteristics of the energy system, and to a certain extent the IT infrastructure at place. If the requirement is to preserve the current organisational structure of energy markets, with separate MOs for different carriers, then the market clearing  MC2 should be implemented in practice, despite the trade-offs related to the additional information exchange between them and additional coordination among the MOs to recover the clearing prices. If a minor change in the current organisational structure, in the form of adding an additional market operator to deal with the conversion orders, is acceptable, then market clearing MC3 should be implemented, as the market clearing prices can be directly retrieved from the  information exchanged among the market operators. Similarly, if the important to be able to resolve the price and volume indeterminancies in the market outcome, or to reduce the information exchange volume, market clearing MC1 should be implemented.

\section{Numerical Results}
\label{section:results}

Implementation of the market clearing is made in Julia (version 1.3.1) using the embedded algebraic modelling language package 'JuMP.jl'\cite{DunningHuchetteLubin2017}, version 0.21.2. As performances are not critical here, the underlying linear programming solvers used are GLPK \cite{GLPK} and Coin CLP \cite{CLP} (the second is known to be substantially faster).

We present here examples and numerical experiments documenting the added value of the novel orders, and of the multi-carrier energy markets introduced in the present work. This Section on numerical results is divided in two parts. 

In Section \ref{sec:numerical-results:conversion-to-elementary}, we describe numerical experiments showing how conversion (resp. storage) orders  enable conversion (resp. storage) technology owners to avoid market risks due to uncertainties on market prices of the origin and destination energy carrier markets of their conversion technology, or to avoid market risks due to uncertainties on market prices of a single energy carrier at different time instances (for storage technology owners). For this purpose, two cases are compared, where conversion (resp. storage) technology owners respectively have or do not have access to the conversion (resp. storage) orders to describe their trading preferences. When conversion (resp. storage) orders are not available, as it is the case in non-integrated energy carrier markets (or for storage orders in the e.g. current electricity markets), traders have to design a bidding strategy that relies on elementary orders and face the associated market risks.

Let us note that welfare objective values obtained as a result of market clearing with conversion and storage orders cannot be directly compared with welfare objective values obtained when conversion and storage technologies must make use of elementary orders in the absence of the adequate type of orders. The reason is that solutions obtained when markets are cleared with only elementary orders  may not be technically feasible for the conversion and storage technologies: this happens for example when a conversion technology has not bought enough gas via gas elementary orders, compared to the electricity to be delivered according to the acceptance of its electricity elementary orders. In that case, the outcome when only elementary orders are used is technically infeasible for the conversion technology. Moreover, as no procurement costs for gas are accounted for in relation to the revenues from the electricity sold, the overall profit of the conversion technology is a priori overestimated and `ad hoc' corrections of the total economic surpluses of all market participants are then needed to obtain figures comparable to welfares obtained when conversion (resp. storage) orders can be used, where technical feasibility is always ensured by the definition of conversion and storage orders. In view of the already clear deviation from an ideal competitive equilibrium outcome shown in Tables \ref{table:num-results:conv_to_elem_elec} \& \ref{table:num-results:storage_to_elem_elec} \, and the above remarks, we have opted not to numerically compare welfares in both cases (decoupled energy carrier markets versus integrated multi-carrier energy  markets).

In Section \ref{sec:numerical-results:generalstats}, we show how the proposed market clearing algorithms could be exploited to better understand the market evolutions due to an evolving energy landscape where new storage and conversion technologies are emerging and evaluating their value in the energy markets.

Ten realistic test cases are built. To make test cases as realistic as possible, historical data or relevant source was used to generate bids with relevant parameters as is explained in the following. 

To generate realistic elementary orders for day-ahead electricity markets, we use historical Italian orders corresponding to the bidding zone `Northern Italy' that are made available by the Italian Energy Exchange GME\footnote{https://www.mercatoelettrico.org/En/Default.aspx}. The historical bids are retrieved from the GME website using the tool provided by the authors of \cite{Savelli2018} to reproduce their numerical simulations, see in particular the associated public repository \cite{openDAM}.

To generate realistic gas orders, gas market data is retrieved from GME website (see \cite{GME}), and was used to generate random elementary gas orders with prices sampled from the publicly available market data. 

\textcolor{black}{For heat orders, three basic orders are assumed, as heat markets are currently not as developed as electricity and gas markets. Two demand orders with respectively volumes of 100MWh and 50MWh, and prices of 2000\euro/MWh and 1000\euro/MWh, have been considered. Besides this, there is a supply order with a volume of 200MWh and a price of 200\euro/MWh. The limit price of supply models a penalty in case of shortage not allowing conversion orders to provide all of the demanded heat, and the volume of the supply order guarentees that there is always enough heat supply. As long as conversion orders are available, as the origin carrier markets have substantially lower prices than the heat supply order, conversion to heat will be preferred by the welfare maximisation problem and this will also reflect in heat market prices that will then be much lower than the upper bound of 200 \euro/MWh.}

Conversion orders are generated using realistic (technology) parameter values for representative conversion technologies available in Italy, following the information that can be found for example in \cite{assessmentcogeneration2016}.

Finally, for storage orders, two large storage units are considered: at each period, for each unit, 500MWh can be stored or 400MWh injected back into the market. The initial storage level is 500MWh for the first storage unit (e.g. hydro-storage or a portfolio of storage devices), and 600MWh for the second storage unit. Losses when energy is stored is of 20\% for the first unit, and of 10\% for the second unit. Losses when energy is injected back into the market is of 30\% for the first unit, and of 20\% for the second unit. Finally, for both units, the cost of storing 1 MWh is 2 \euro, which is added to the price paid when buying the energy, modelling marginal storage costs or a minimum profit to recover.

\subsection{Elimination of market risk for conversion and storage technologies} \label{sec:numerical-results:conversion-to-elementary}

The objective in this section is first to \emph{illustrate} the market risks encountered by conversion technology owners that do not have access to  multi-carrier energy markets integrated via conversion orders. It shows the added value of the conversion orders which basically \emph{eliminate} risks associated with forecasting market price differences between the energy carrier markets. It will be shown that conversion technology owners cannot avoid losses or opportunity costs if they can only rely on elementary orders. Similar observations are made regarding the availability to storage technology owners of storage orders, illustrating elimination of risks associated with forecasting single carrier market prices at different time instances. In addition, possible technical infeasibilities due to the implementation of the market clearing result in the absence of storage orders for the storage technology owner are discussed.

For that purpose, we compare market outcomes (a) with conversion orders, (b) with conversion orders replaced by elementary orders used by conversion technology owners to bid in the separated carrier market setting. Similar comparisons are made regarding market outcomes (a) with storage orders, (b) with storage orders replaced by elementary orders used by storage technology owners to bid in markets without storage orders.

This necessitates to make some basic assumptions to model a \emph{basic} bidding behaviour of a conversion technology owner bidding for example in separated gas and heat markets, and similarly for a storage technology owner in (electricity) markets without storage orders. The basic bidding behaviour is sufficient to illustrate the benefits of the proposed integrated multi-carrier market clearing. Considering advanced trading strategies (e.g.based on stochastic optimisation) is beyond the scope of the present work. 

The following basic bidding behaviour is considered for conversion technology owners. It is assumed that a conversion technology owner makes use of forecasts of market prices in the origin and destination carrier market of its conversion technology. For example, a gas-fired power plant owner will forecast gas and electricity market prices for the following day, to assess whether the price difference between gas and electricity is financially attractive. The price difference should be sufficiently high, i.e., cover variable operations and maintenance costs, losses due to the imperfect conversion efficiency, plus some minimum marginal profit. However, the forecasts are imperfect. The perfect prices are the prices that would be obtained in the integrated energy carrier markets case, representing the competitive equilibrium prices that would be reached in a perfectly competitive environment. \textcolor{black}{Market forecasts are simulated by starting from those ideal market prices and adding a forecast error. 
The forecast errors are assumed to follow an uniform distribution between -5\% and +5\%.} If the forecasts show that the price difference is large enough for participants to make a positive profit, the participants submit a demand order in the origin carrier market with a price slightly higher (by 0.1 \euro/MWh) than the forecasted price (while still not offering too much), to try to be fully accepted. In the destination carrier market, knowing the highest price at which gas would have been bought (the limit price submitted), the limit sell price for electricity is chosen so as to guarantee that a non-negative profit is obtained from the buy and sell operations associated to the conversion (plus a small margin of 0.1\euro/MWh). However, it could happen that too much gas is bought in the origin market, and no electricity is finally accepted in the destination market, or the other way around. 

The following basic bidding behaviour, formed as a two-step procedure, is considered for storage technology owners. In the first step, we generate electricity prices in the presence of the original storage orders, and then add to those prices an error following a uniform distribution ranging from -5\% to 5\%, in order to represent the forecast prices. In the second step, we solve a profit maximisation problem of the storage against the forecast prices to generate the optimal charging/discharging energy volumes under the forecast prices. The forecast prices plus/minus a margin of 0.1 \euro/MWh, and the optimal charging/discharging quantities are used to create elementary orders.

Although the bidding behaviour of market participants has an impact on market prices, it is assumed that they do not have market power and that the impact of their bidding behaviour on market prices is negligible, i.e., they are price takers, as can be observed on Figure~\ref{fig:num-results:conv_to_elem_elec} below. The electricity and gas price dynamic is here not significantly impacted by the basic bidding behaviour described above, i.e. the translation of conversion orders to elementary orders.

\begin{figure}[ht!]
    \centering
    \includegraphics[scale=0.45]{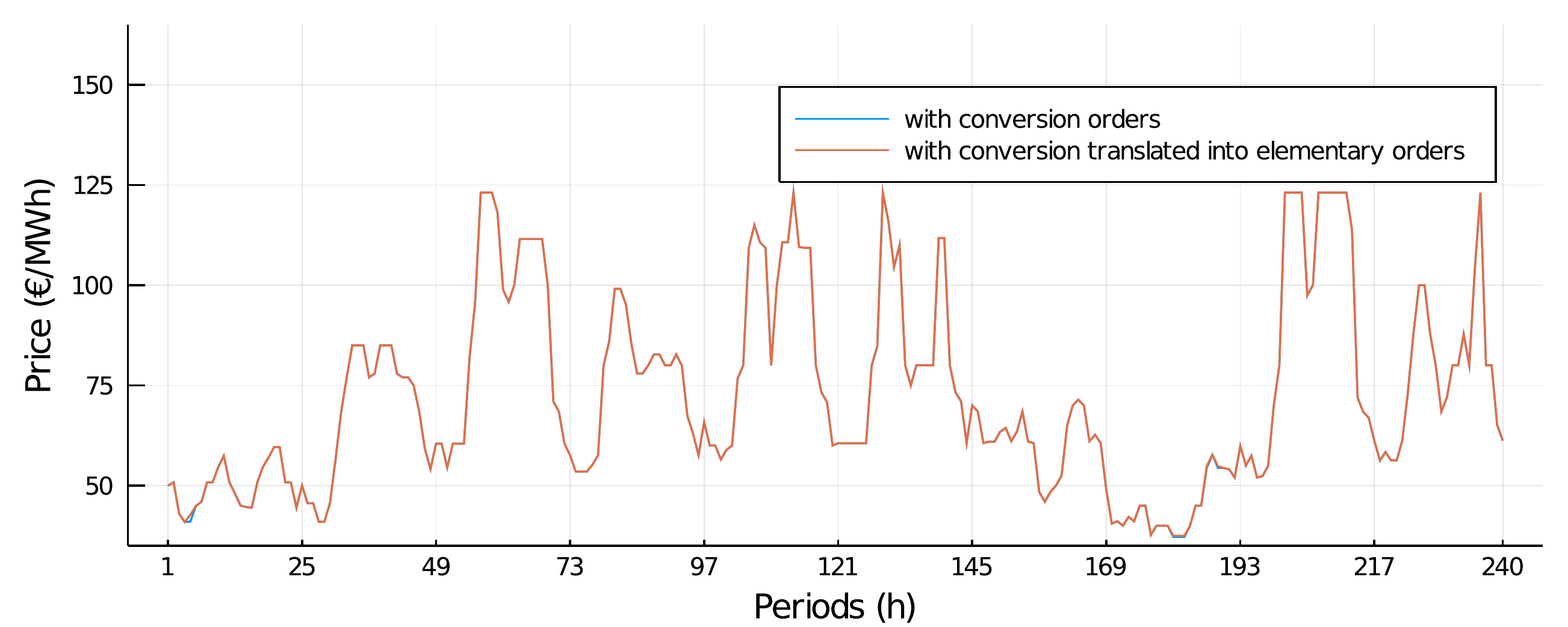}
    \caption{Impact on electricity price dynamic of translating conversion orders to elementary orders}
    \label{fig:num-results:conv_to_elem_elec}
\end{figure}
\begin{table}
\centering
\caption{Deviations from equilibrium when conversion technologies need to use elementary orders (no conversion orders). Conversion technologies face losses or opportunity costs.}
\label{table:num-results:conv_to_elem_elec}
 \begin{tabular}{||c|c|c|c||} 
 \hline \hline
  \textbf{Date} & \textbf{Losses (K\euro{})}  & \textbf{Opportunity Costs (K\euro{})}\\ [0.5ex] 
 \hline\hline
    2020-03-01 & 0.08  & 0.12 \\ 
\hline
    2020-03-02 & 5.09  & 5.20 \\ 
\hline
    2020-03-03 & 2.09  & 2.09 \\ 
\hline
    2020-03-04 & 0.00  & 0.00 \\ 
\hline
    2020-03-05 & 3.16  & 0.90 \\ 
\hline
    2020-03-06 & 4.55  & 4.62 \\ 
\hline
    2020-03-07 & 2.66  & 1.84 \\ 
\hline
    2020-03-08 & 6.29  & 6.34 \\ 
\hline
    2020-03-09 & 1.00  & 1.03 \\ 
\hline
    2020-03-10 & 6.53  & 6.61 \\
 \hline  \hline
\end{tabular}
\end{table}

As can be seen in Table \ref{table:num-results:conv_to_elem_elec}, conversion technologies having to use elementary orders face losses or opportunity costs. \textcolor{black}{Losses are caused by too much gas bought in the origin carrier market for which not enough electricity is sold in the destination market. Opportunity costs are incurred if the final price spread between the gas and electricity markets is attractive in terms of the conversion costs and efficiency, but using elementary orders to buy gas and sell electricity did not enable to obtain the best possible profits. Opportunity costs are computed as the maximum profit a conversion technology could have obtained given the price spread, minus the profits actually made. } 

\begin{figure}[ht!]
    \centering
    \includegraphics[scale=0.45]{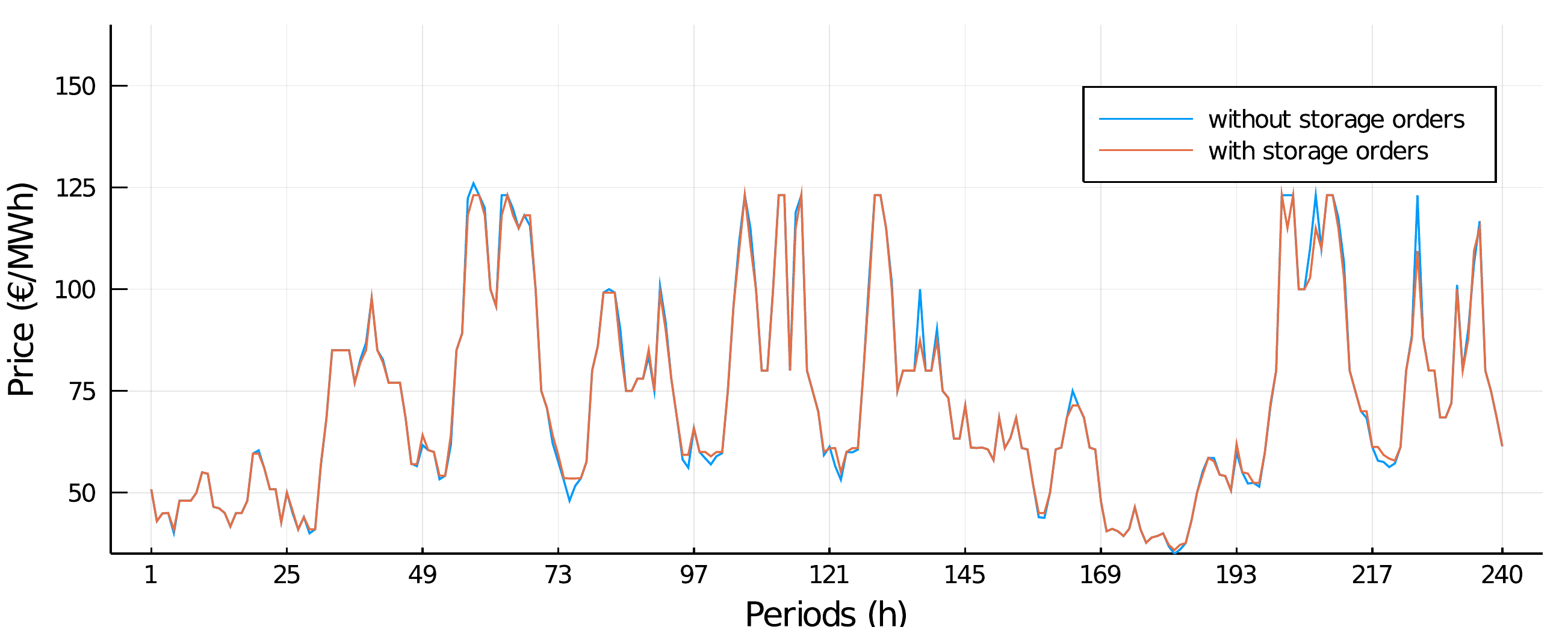}
    \caption{Comparison of electricity prices in the case without storage orders and the case with storage orders.}
    \label{fig:price_electricity_full_and_nots}
\end{figure}

Figure~\ref{fig:price_electricity_full_and_nots} compares electricity prices in the case with storage orders and the case where storage orders are replaced by elementary orders following the procedure described above. It is observed that storage orders contribute to reducing price spikes. In addition, we identify one drawback when there are no storage orders in the market, and the storage participates in the market via elementary orders. As shown in Figure~\ref{fig:energy_storage}, the energy storage level could become negative, i.e. it would have to deliver energy which is not available (stored).  This indicates the possible operational infeasibility of the storage if the initial energy level is too low, due to the fact that elementary orders are insufficient in accounting for the technical specifications of the storage, especially inter-temporal constraints. Note that this is an illustration of the current participation of storage technology in the markets, which would not necessarily take place in practice, as we only considered day-ahead markets, and did not consider a cutting-edge bidding strategy. In case of operational infeasibility, after the clearing of the day-ahead market, the storage owner still has a possibility e.g.  to trade on the intraday market or to be exposed to the imbalance prices, which is out of scope of the paper. 

\begin{figure}[ht!]
    \centering
    \includegraphics[scale=0.45]{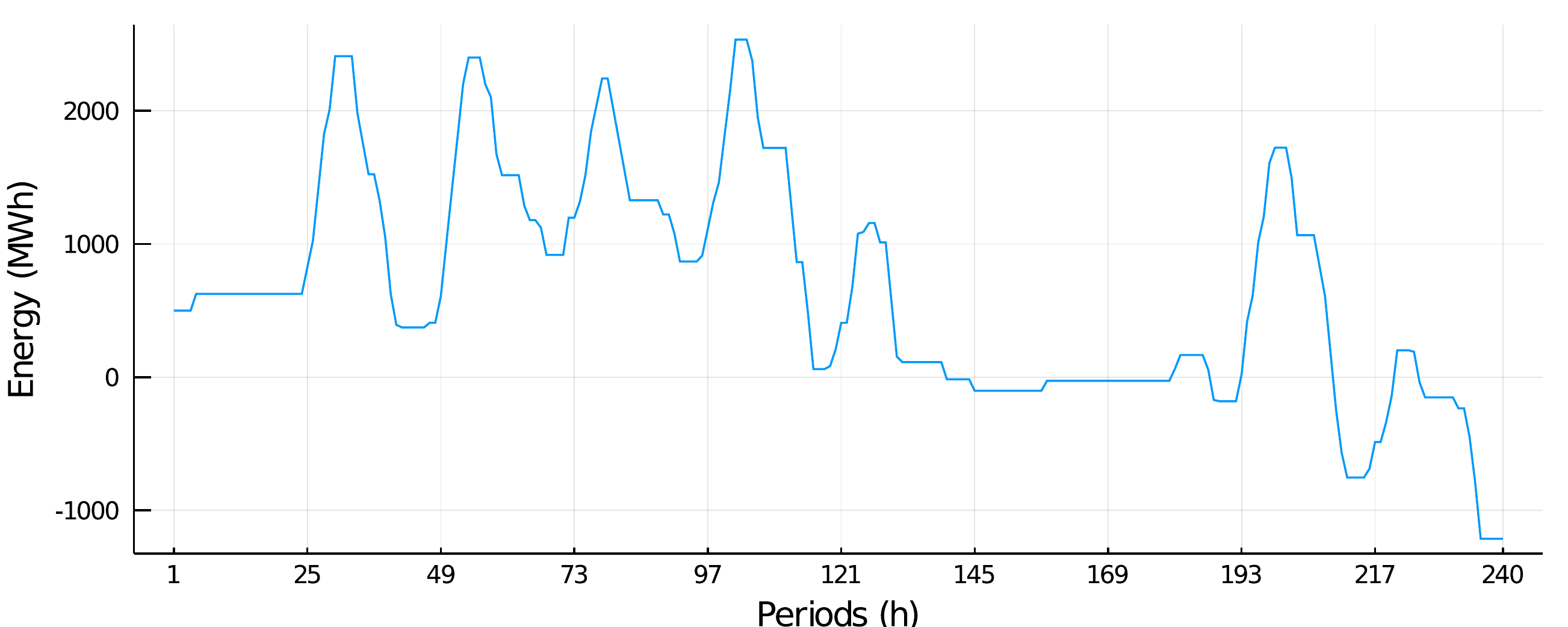}
    \caption{Energy dynamics of the storage when participating in the market via elementary orders.}
    \label{fig:energy_storage}
\end{figure}

\begin{table}
\centering
\caption{Deviations from equilibrium when storage technologies need to use elementary orders (no storage orders)}
\label{table:num-results:storage_to_elem_elec}
 \begin{tabular}{||c|c|c||} 
 \hline \hline
  \textbf{Date} & \textbf{Value Overbought Energy (K\euro{})}  & \textbf{Opportunity Costs (K\euro{})}\\ [0.5ex] 
 \hline\hline
    2020-03-01 & 11.4  & 5.26 \\ 
\hline
    2020-03-02 & 2.75  & 3.82 \\ 
\hline
    2020-03-03 & 37.29  & 1.27 \\ 
\hline
    2020-03-04 & -15.16  & 3.28 \\ 
\hline
    2020-03-05 & -53.79  & 2.48 \\ 
\hline
    2020-03-06 & -11.92  & 3.72 \\ 
\hline
    2020-03-07 & -0.17  & 6.78 \\ 
\hline
    2020-03-08 & -12.56  & 4.69 \\ 
\hline
    2020-03-09 & -32.19  & 3.16 \\ 
\hline
    2020-03-10 & -48.38  & 3.45 \\
 \hline  \hline
\end{tabular}
\end{table}

Table \ref{table:num-results:storage_to_elem_elec} shows the market imperfection if storage technologies have to participate in the market via elementary orders. The second column describes the  monetary value corresponding to the energy overbought until the end of the day. \textcolor{black}{To estimate the monetary value, we first calculate the difference of the energy level at the end of the day with the initial energy level of the same day to get the quantity of overbought energy, and then the price of the overbought energy is assumed to be the electricity market price in the last hour of the day. The product of the overbought energy quantity and electricity price yields the monetary value.}  The third column indicates the profit loss of the storage technology due to the lack the storage orders. \textcolor{black}{In order to calculate the profit loss,  we first calculate the profit of each day in the case study with storage orders and the case study without storage orders (the value of overbought energy is accounted for). Then the difference of the profits from two case studies are reported as the opportunity costs.} This table demonstrates that the storage can always yield a higher profit if he can participate in the market via storage orders.

\subsection{Impact of conversion and storage technologies on price formation in the integrated carrier markets} \label{sec:numerical-results:generalstats}

We now simulate the effect of adding new conversion and storage technologies in the market. Since the tools introduced above allow to reach a competitive equilibrium when clearing markets with conversion or storage orders, these order types can be used to represent the ideal market impact of conversion or storage technologies in a market under the assumption of perfect competition. These orders can hence be used to evaluate the market impacts of adding new technologies on the market clearing volumes and prices, and to calculate expected day-ahead market revenues for new technologies.

As expected, conversion orders tend to lower prices in the destination carrier market and to increase prices in the origin carrier market, see Figure \ref{fig:num-results:conv_elec}. Note that conversion orders could lead to a full price convergence between the different carrier markets: for that, it is sufficient to have (a) all conversion orders with a conversion efficiency of 100\%, (b) no conversion costs, and (c) a total conversion capacity which is `not scarce' in the sense that all the conversion capacity from a given origin carrier market to a given destination carrier market is not fully used in the market outcome. 

\begin{figure}[ht!]
    \centering
    \includegraphics[scale=0.45]{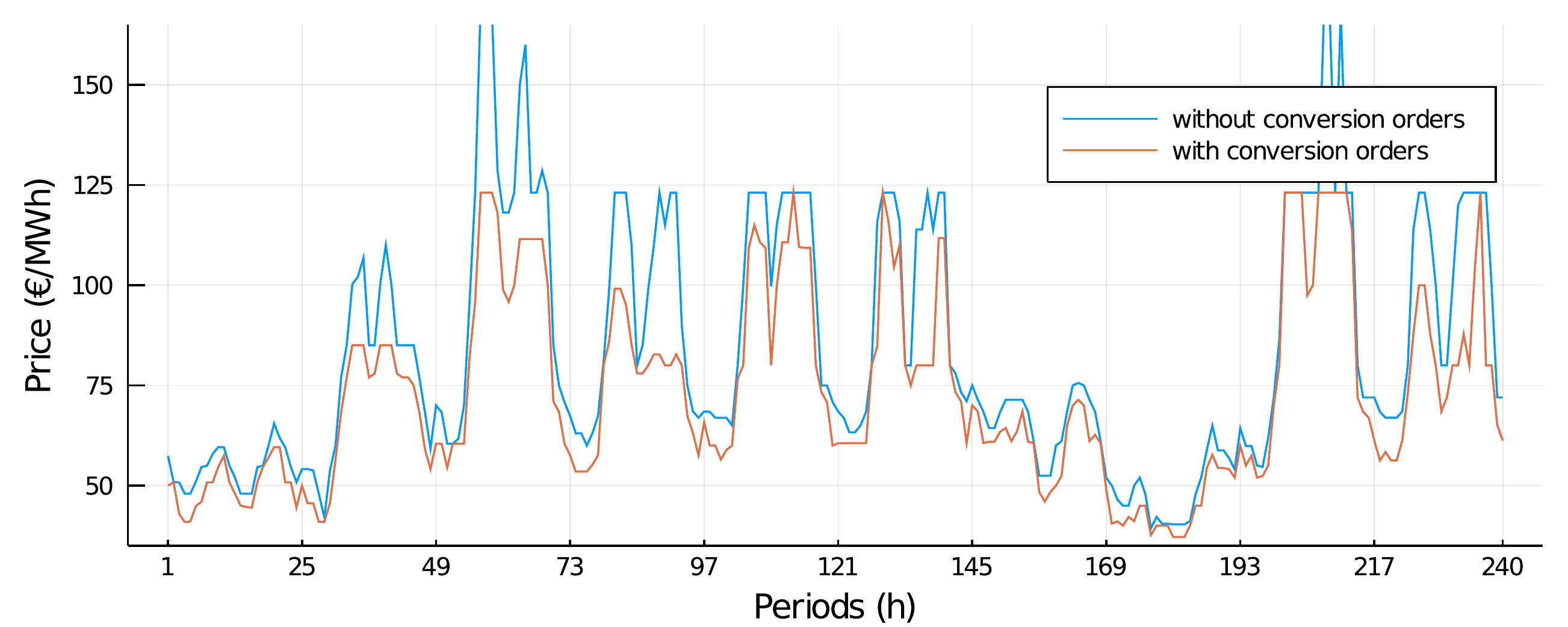}
    \caption{Impact of adding conversion orders on electricity market prices: conversion capacity from gas cheaper than electricity enables to lower electricity prices.}
    \label{fig:num-results:conv_elec}
\end{figure}

Additional storage orders on their side tend to shave electricity price peaks, as can be seen on Figure~\ref{fig:num-results:ts_elec}. The intuitive explanation is as follows: electricity is bought in periods when it is cheaper and sold at more expensive periods allowing to avoid the recourse to the most expensive peak units. The effect is to (substantially) lower the price peaks at periods of peak load, at the expense of a small price increase in periods where electricity is cheaper.

\begin{figure}[ht!]
    \centering
    \includegraphics[scale=0.45]{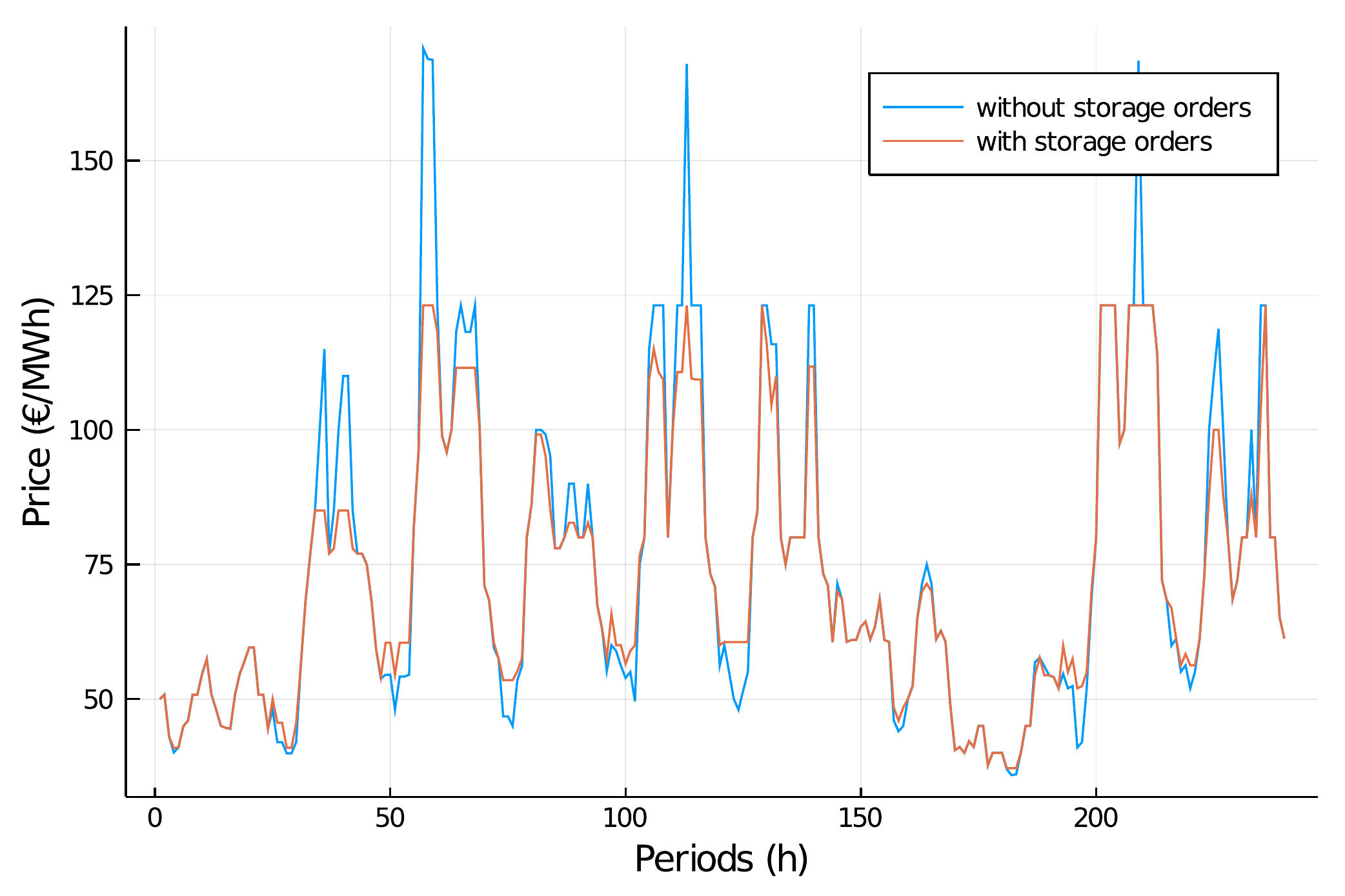}
    \caption{Smoothing effect of storage orders in electricity market prices.}
    \label{fig:num-results:ts_elec}
\end{figure}

\vspace{-0.5cm}

\section{Conclusion and Future Work}
\label{section:conclusion}

Sector coupling in general, and hence integration multi-energy-systems at the market level, are promising venues to cost efficient and sustainable energy systems of the future. As a result, there is a need for a well coordinated and explicit integration of energy carrier markets with the goal of gaining higher flexibility in the operation of underlying energy systems and improving the economic efficiency of the related energy markets. Toward this goal, we introduced integrated (day-ahead) multi carrier energy markets, featuring novel conversion and storage orders allowing market participant to adequately represent their conversion or storage technologies when trading energy in these markets. We have analytically and numerically shown how \textcolor{black}{such markets} alleviates market risk for traders for whom there is no more need to internalise market risk mitigating measures into their bidding strategies. Next to the novel order types, we have shown how novel constraint types such as pro-rata and cumulative constraints, can be used to better represent the techno-economic constraints of market participants.

In a second stage, we have shown that integrated day-ahead multi-carrier energy markets can be cleared in three different ways, each with different organisational structure. In a centralised integrated multi-carrier energy market, \textcolor{black}{there is a single market operator that has all the information about the orders and constraints and clears all the carrier markets together. Alternatively, two decentralised integrated day-ahead multi-carrier energy  markets are introduced.} They differ in the way decentralisation is performed and subsequently in the implications in terms of market operators that are involved and the information exchange among them. \textcolor{black}{One decentralised market clearing can be implemented by assigning a single market operator per energy carrier market. Another decentralised market clearing requires a  single  market  operator  per  carrier  market and an additional conversion order market operator.} The two decentralised market clearings also differ in the involved Lagrangian relaxations solved via a subgradient method, and consequently, the information that needs to be shared between the decomposed market operators. In both cases, the approach consists in solving the welfare maximisation program by solving a Lagrangian dual and recovering optimal primal solutions. In the decentralised market clearings, it is not possible to resolve price or volume indeterminancies, as opposed to the centralised market clearing case.

To conclude, which market clearing procedure is superior among the three largely depends on the  policy goals to be achieved. If the requirement is to preserve the current organisational structure of energy markets, with separate MOs for different carriers, then a decentralised market clearing should be implemented in practice, despite the trade-offs related to the additional information exchange between them. If the current market organisation should be perfectly preserved, the proposed decentralised market clearing MC2 should be implemented, even though \textcolor{black}{the market clearing price recovery is not straightforward. Then, if a slight change to the current structure, in the form of adding an additional market operator to deal with the conversion orders, is acceptable, then market clearing MC3 should be implemented, given that market clearing prices can be directly retrieved from the  information exchanged among the market operators.}
Similarly, if it is important to resolve the price and volume indeterminancies in the market outcome, or to reduce the information exchange volume, a centralised market clearing should be the pathway towards integration of multi-carrier energy markets.

In a sequel to this paper, the authors would like to leverage this work to quantitatively analyse interactions between storage and conversion technologies and how they benefit from each other in an ideal, integrated multi-carrier energy market. Integration of advanced bidding products with integer constraints also raise well-known pricing issues as they give rise to so-called markets with non-convexities where most of the time, a competitive equilibrium supported by uniform prices does not exist \cite{oneill2005, madani2018revisiting}: studying these pricing questions in the frame of multi-carrier energy markets is also an interesting venue for extension of the present work. \textcolor{black}{Another avenue for further research would be to investigate the splitting of conversion costs for conversion technologies with non-separable single input, multiple output energy carriers (e.g., CHPs with non-separable gas to heat and electricity conversion costs).} 

\vspace{-0.4cm}

\section*{Acknowledgement}
The work presented in this paper was carried out in  the MAGNITUDE project which has received funding from the European Union's Horizon 2020 research and innovation programme under grant agreement No 774309. This paper and the results described reflect only the authors' view. The European Commission and the Innovation and Networks Executive Agency (INEA) are not responsible for any use that may be made of the information they contain. 

\textcolor{black}{The authors would like to thank the member of MAGNITUDE project consortium for their support, valuable comments and feedback.}

\vspace{-0.4cm}
\section*{Declaration of competing interest}
\textcolor{black}{The authors declare that they have no known competing financial interests or personal relationships that could have appeared to influence the work reported in this paper.}

\vspace{-0.2cm}

\appendix
\renewcommand{\thesection}{\Alph{section}}

\section{Welfare maximization and competitive equilibrium: the general case}

We show here in the the context of abstract general orders (and constraints) how maximising welfare leads to a competitive equilibrium. This is a generalisation of the celebrated result by Paul Samuelson linking spatial price equilbria and linear programming used to maximize welfare \cite{samuelson}. Notation here follows O'Neill et al. in \cite{oneill2005} where the authors have generalised the results to a non-convex setting where non-convexities come from integer variables used to describe the techno-economic constraints of market participants. In what follows, we assume that the involved feasible sets are non-empty and bounded.

Compared to \cite{oneill2005}, in our more classic convex setting, we provide below a simple straightforward proof of  Theorem \ref{theorem:competitive-equilibrium} which directly follows from strong duality for linear programs.

In the following abstract welfare maximization problem, $x_k$ is the vector of decisions related to market participant $k$, constraints \eqref{generalmodel-balance} represent balance constraints (with optimal dual variables providing market prices), and constraints \eqref{generalmodel-individual-constraints} represent individual constraints per market participant $k$

\begin{align}
    \max_{x} &\sum_k c_k^T x_k, \label{generalmodel-objective}\\
 \text{s.t.}   & \sum_k A_{k} x_k  = 0, &  [\pi]\label{generalmodel-balance} \\
    & B_{k} x_k  \leq b_k  & \forall k  \label{generalmodel-individual-constraints} 
\end{align}

\vspace{-0.3cm}

\begin{definition}[Competitive (or Walrasian) equilibrium]
\label{definition:competitive-equilibrium}
A competitive equilibrium consists in decisions $(x_k^*)_{k \in K}$ and prices $\pi^*$ such that:
\begin{itemize}
    \item for each $k$, for the given fixed prices $\pi^*$,  $(x_k^*)$ solves 
    \begin{equation}
        \max_{x_k} c_k^T x_k  - (\pi^*)^T (A_{k} x_k) ,
        \label{WalrasianEQ_begin}
    \end{equation}
    subject to 
    \begin{align}
        & B_{k} x_k  \leq b_k, \label{WalrasianEQ_end}
    \end{align}
    \item The market clears: $\sum_k A_{k} x_k^* = 0$, cf. condition (\ref{generalmodel-balance}) above.
\end{itemize}
\end{definition}

\vspace{-0.4cm}

\begin{theorem}[Generalisation of the main result in \cite{samuelson} and special case without binary decisions of Theorem 2 in \cite{oneill2005}, with simplified proof]
\label{theorem:competitive-equilibrium}
Let us consider an optimal solution $(x_k^*)_{k \in K}$ to the welfare maximisation problem (\ref{generalmodel-objective})-(\ref{generalmodel-individual-constraints}). Let $\pi^*$ be obtained as (linear programming) optimal dual variables to the constraints \eqref{generalmodel-balance} in the linear optimisation problem \eqref{generalmodel-objective}-\eqref{generalmodel-individual-constraints}. The solution $(x_k^*)_{k \in K}$ and the prices $\pi^*$ form a competitive equilibrium as defined in Definition \ref{definition:competitive-equilibrium}.
\end{theorem}

\begin{proof}
Balance constraints \eqref{generalmodel-balance} are directly satisfied by $(x_k^*)_{k \in K}$ solving (\ref{generalmodel-objective})-(\ref{generalmodel-individual-constraints}). It remains to show that $(x_k^*)_{k \in K}$ also solves \eqref{WalrasianEQ_begin}-\eqref{WalrasianEQ_end}.

This follows from the fact, proved below, that $\pi^*$ and $(x_k^*)_{k \in K}$ solve the following Lagrangian dual problem, or ``partial linear programming dual", of \eqref{generalmodel-objective}-\eqref{generalmodel-individual-constraints} where the balance  conditions \eqref{generalmodel-balance} are dualized using Lagrangian multipliers $\pi$:

\begin{equation} 
\min_{\pi} \Big\{ \max_x \sum_k c_k^T x_k-\pi^T(A_kx_k),\ s.t.\ B_kx_k \leq b_k \ \forall k\Big\} = \min_{\pi} \Big\{\sum_k \{\max_{x_k}   c_k^T x_k  - \pi^T (A_{k} x_k), ~\text{s.t.}\ B_{k} x_k  \leq b_k \}  \Big\}.\label{partial-dual}
\end{equation} 

The equality in \eqref{partial-dual} follows from the fact that the inner maximisation problem in the left-hand side can be separated per market participant $k$. It can then be seen that in the right-hand side, for the $\pi$ fixed, the inner problems are exactly \eqref{WalrasianEQ_begin}-\eqref{WalrasianEQ_end} written for each participant $k$, that must be solved by $(x_k^*)_{k \in K}$ to obtained the desired result.

Let us verify that $(x_k^*)_{k \in K}$ indeed solve the inner maximisation problems in \eqref{partial-dual}. By strong duality for linear programs (see  Theorem 1 in \cite{geoffrion1974lagrangean} and \cite{geoffrion1971duality}, Section 6.1 for more details on strong duality for such a 'partial dual'), $\pi^*$ solves the left-hand side of \eqref{partial-dual}, and:
\begin{equation} 
\sum_k c_k^T x_k^* =  \min_{\pi} \Big\{\max_{x} \sum_k  \Big(c_k^T x_k  - \pi^T (A_{k} x_k)\Big), ~\text{s.t.}\ B_{k} x_k  \leq b_k \ \forall k  \Big\}. \label{strong-duality-partial-dual}
\end{equation} 

Using \eqref{generalmodel-balance} multiplied by $\pi$ and the fact that $\pi^*$ solves \eqref{partial-dual}, we have:

\begin{equation} 
\sum_k c_k^T x_k^* -  (\pi^*)^T (A_{k} x_k) =  \sum_k \Big(\max_{x_k}   c_k^T x_k  - (\pi^*)^T (A_{k} x_k), ~\text{s.t.}\ B_{k} x_k  \leq b_k \Big). \label{strong-duality-partial-dual2}
\end{equation} 

Since for all $k$, $x_k^*$ satisfies $B_k x_k \leq b_k$, i.e. the $x_k^*$ are feasible for the profit maximisation problems in the right-hand side of \eqref{strong-duality-partial-dual2}, they must also be optimal solutions for these problems: otherwise, \eqref{strong-duality-partial-dual2} would not hold and the right-hand-side would be strictly greater than the left-hand side, contradicting strong duality for linear programs. \end{proof}

\bibliography{mybibfile}

\end{document}

%% file: Definitions_and_Packages.tex
\usepackage{lineno,hyperref}
\modulolinenumbers[5]

\usepackage[utf8]{inputenc}
\usepackage{amsmath,amsthm,amssymb,amsfonts}
\usepackage{a4wide}
\usepackage{array, graphicx} 
\usepackage[parfill]{parskip} 
\usepackage[official]{eurosym}

\usepackage{soul}
\usepackage{enumitem}
\usepackage{xcolor}
\usepackage{glossaries}
\usepackage{float}
\usepackage{caption}
\usepackage{hyperref}
\usepackage{xcolor}
\usepackage{enumerate}
\usepackage{enumitem}
\usepackage{adjustbox}
\usepackage{multirow}
\usepackage{nomencl}
\usepackage{etoolbox}
\usepackage{comment}
\usepackage{longtable}
\usepackage{cancel}
\usepackage{rotating}

\usepackage{stfloats}
\usepackage{multicol}
\usepackage{nccmath}
\usepackage{algorithmic}
\usepackage{array}

\usepackage{caption}
\usepackage{subcaption}

\usepackage{fixltx2e}
\usepackage{url}

\usepackage{authblk}

\newtheorem{theorem}{Theorem}[section]

\newtheorem{proposition}{Proposition}
\newtheorem{definition}{Definition}
\everymath{\displaystyle}

\newcommand{\myset}[1]{\mathcal{#1}}

\DeclareMathOperator{\Lagr}{\mathcal{L}} 

\graphicspath{{images/}{../images/}}

\newcommand{\PreserveBackslash}[1]{\let\temp=\\#1\let\\=\temp}
\newcolumntype{C}[1]{>{\PreserveBackslash\centering}m{#1}}
\newcolumntype{R}[1]{>{\PreserveBackslash\raggedleft}m{#1}}
\newcolumntype{L}[1]{>{\PreserveBackslash\raggedright}m{#1}}